\documentclass[12pt,reqno]{amsart}
\usepackage[T1]{fontenc}
\usepackage[dvipsnames]{xcolor}
\usepackage{float}
\usepackage[utf8]{inputenc}
\usepackage{amsmath,amsthm,amssymb}
\usepackage{indentfirst}
\usepackage{url}
\usepackage{tikz-cd}
\usepackage[colorlinks]{hyperref}
\usepackage{colonequals}
\usepackage{graphicx}
\usepackage{enumerate}
\usepackage{soul}
\usepackage{stmaryrd}
\usepackage{mathrsfs,mathtools}
\usepackage{palatino}
\usepackage{bbm}
\usepackage[alphabetic,nobysame]{amsrefs}
\usepackage{hyperref}
\usepackage{ytableau}
\usepackage[margin=0.8in]{geometry}

\newtheorem{thm}{Theorem}[section]
\newtheorem{defn}[thm]{Definition}
\newtheorem{example}[thm]{Example}
\newtheorem{lem}[thm]{Lemma}

\newtheorem{prop}[thm]{Proposition}

\newtheorem{rem}[thm]{Remark}

\newcommand{\Z}{\mathbb{Z}}
\newcommand{\x}{\textbf{x}}
\newcommand{\y}{\textbf{y}}

\newcommand{\MBPD}{\mathsf{MBPD}}
\newcommand{\RPD}{\mathsf{RPD}}

\newcommand{\SRPD}{\mathsf{SRPD}}

\newcommand{\fG}{\mathfrak{G}}
\newcommand{\fS}{\mathfrak{S}}
\newcommand{\fCM}{\mathfrak{CM}}
\newcommand{\rwt}{\mathsf{rwt}}
\newcommand{\cwt}{\mathsf{cwt}}
\newcommand{\rw}{\mathsf{rw}}
\newcommand{\cw}{\mathsf{cw}}
\newcommand{\rajcode}{\mathsf{rajcode}}

\newcommand{\Snow}{\mathsf{Snow}}
\newcommand{\Rothe}{\mathsf{Rothe}}
\newcommand{\PD}{\mathsf{PD}}
\newcommand{\dark}{\mathsf{dark}}

\newcommand{\btile}{
	\begin{tikzpicture}[x=0.8em,y=0.8em,thick,color = blue]
		\draw[step=1,gray,thin] (0,0) grid (1,1);
		\draw[color=black, thick, sharp corners] (0,0) rectangle (1,1);
\end{tikzpicture}}

\newcommand{\htile}{
	\begin{tikzpicture}[x=0.8em,y=0.8em,thick,color = blue]
		\draw[step=1,gray,thin] (0,0) grid (1,1);
		\draw[color=black, thick, sharp corners] (0,0) rectangle (1,1);
		\draw(1.0,0.5)--(0.0,0.5);
	\end{tikzpicture}
}

\newcommand{\vtile}{
	\begin{tikzpicture}[x=0.8em,y=0.8em,thick,color = blue]
		\draw[step=1,gray,thin] (0,0) grid (1,1);
		\draw[color=black, thick, sharp corners] (0,0) rectangle (1,1);
		\draw(0.5,1.0)--(0.5,0.0);
\end{tikzpicture}}

\newcommand{\ptile}{
	\begin{tikzpicture}[x=0.8em,y=0.8em,thick,color = blue]
		\draw[step=1,gray,thin] (0,0) grid (1,1);
		\draw[color=black, thick, sharp corners] (0,0) rectangle (1,1);
		\draw(0.5,1.0)--(0.5,0.0);
		\draw(1.0,0.5)--(0.0,0.5);
\end{tikzpicture}}

\newcommand{\rtile}{
	\begin{tikzpicture}[x=0.8em,y=0.8em,thick,color = blue]
		\draw[step=1,gray,thin] (0,0) grid (1,1);
		\draw[color=black, thick, sharp corners] (0,0) rectangle (1,1);
		\draw(1.0,0.5)--(0.5,0.5)--(0.5,0.0);
\end{tikzpicture}}

\newcommand{\jtile}{
	\begin{tikzpicture}[x=0.8em,y=0.8em,thick,color = blue]
		\draw[step=1,gray,thin] (0,0) grid (1,1);
		\draw[color=black, thick, sharp corners] (0,0) rectangle (1,1);
		\draw(0.5,1.0)--(0.5,0.5)--(0.0,0.5);
\end{tikzpicture}}

\newcommand{\mtile}{
	\begin{tikzpicture}[x=0.8em,y=0.8em,thick,color = blue]
		\draw[step=1,gray,thin] (0,0) grid (1,1);
		\draw[color=black, thick, sharp corners] (0,0) rectangle (1,1);
		\draw(0.5,1.0)--(0.5,0.5)--(0.0,0.5);
		\node at (0.5,0.5) {\tiny$\bullet$};
\end{tikzpicture}}

\definecolor{darkblue}{rgb}{0.0,0,0.7} % darkblue color
\definecolor{darkred}{rgb}{0.7,0,0} % darkred color
\definecolor{darkgreen}{rgb}{0, .6, 0} % darkgreen color
\newcommand{\defi}[1]{{\color{darkred}\emph{#1}}} % emphasis of a definition

\begin{document}
	\title{Constructing Maximal Bumpless Pipedreams for Double Grothendieck Polynomials}

\author{Xuanying Han}
	\address{Department of Mathematics\\
Tianjin University of Finance and Economics, Tianjin 300222, P.R. China}
	\email{hanxuanying@stu.tjufe.edu.cn}

\author{Sophie C.C. Sun}
	\address{Department of Mathematics\\
Tianjin University of Finance and Economics, Tianjin 300222, P.R. China}
	\email{sophiesun@tjufe.edu.cn}

	\maketitle
	\begin{abstract}
		Pipedreams and marked bumpless pipedreams are two combinatorial models that compute double Grothendieck polynomials. 
        While studying matrix Schubert varieties, Pechenik, Speyer, and Weigandt defined the Rajchgot code, denoted by \(\rajcode(\cdot)\), which captures the leading monomial of the top-degree component of a Grothendieck polynomial.
        Combinatorially, their result implies that there exists a unique pipedream (or marked bumpless pipedream) with row weight $\rajcode(w)$ and column weight $\rajcode(w^{-1})$. A construction of such a pipedream was subsequently given by Chou and Yu. In this paper, we resolve the marked bumpless pipedream version of this problem by providing an explicit algorithm.
	\end{abstract}
	
	\section{Introduction}
	The matrix Schubert variety $X_w$ of a permutation $w \in S_n$ is a generalized determinantal variety that has been studied extensively~\cite{ful92,km05,kmy09,wy18} due to its close relation to Schubert varieties and to the Schubert cell decomposition of the flag variety. An important invariant of such varieties is the Castelnuovo–Mumford regularity, which measures their algebraic complexity. Since matrix Schubert varieties are Cohen–Macaulay~\cite{ful92,km05,Ram}, their regularity can be computed as the degree difference between the highest and lowest degrees monomials appearing in their K-polynomial. Knutson and Miller~\cite{km05} showed that the K-polynomial of $X_w$ is the Grothendieck polynomial $\fG_w(\x)$. The Grothendieck polynomials are a family of polynomials introduced by Lascoux and Sch\"utzenberger to represent K-theory classes of Schubert varieties~\cite{ls82a,ls82b}. Their lowest degree homogeneous components are the famous Schubert polynomials $\fS_w(\x)$, and their highest degree homogeneous components are the Castelnuovo–Mumford polynomials $\fCM_w(\x)$. Since the degrees of Schubert polynomials are known, computing the Castelnuovo–Mumford regularity of $X_w$ is then reduced to determining the degree of $\fCM_w(\x)$.
	
	With this motivation, there has been a surge in the study of the Castelnuovo–Mumford polynomials $\fCM_w(\x)$~\cite{cy25,hafner22,hafner26,rrr21,rrw23}. The first complete characterization of their degree was given by Pechenik, Speyer, and Weigandt~\cite{psw24}, where they defined the weak composition $\rajcode(w)$ corresponding to the leading monomial of $\fCM_w(\x)$ under reverse lexicographic order. They computed $\rajcode(w)$ by considering increasing subsequences in the one-line notation of $w$. In addition to their formula, the climbing chain models introduced by Dreyer, Mészáros, and St. Dizier~\cite{dms24}, and the snow diagrams introduced by Pan and Yu~\cite{py24} also compute $\rajcode(w)$ combinatorially.
	
	The double Grothendieck polynomials $\fG_w(\x;\y)$ form a refinement of Grothendieck polynomials in equivariant K-theory, with the specialization $\fG_w(\x; 0) = \fG_w(\x)$. Their highest degree homogeneous components are the double Castelnuovo–Mumford polynomials $\fCM_w(\x; \y)$. Pechenik, Speyer, and Weigandt showed that $\rajcode(\cdot)$ also controls the leading monomial of $\fCM_w(\x; \y)$.
    
	\begin{thm}[{\cite[Theorem 1.4]{psw24}}]
		\label{T: PSW}
		For any term order with $x_n > \cdots > x_1$ and $y_n > \cdots > y_1$, the leading monomial of $\fCM_w(\x;\y)$ is $x^{\rajcode(w)} y^{\rajcode(w^{-1})}$. Furthermore, the coefficient of this monomial is $1$.
	\end{thm}
	
	Pipedreams~\cite{bb93,bjs93,fk94} and marked bumpless pipedreams~\cite{lls21,lls23,weigandt21} are certain tilings labeled by permutations that compute $\fG_w(\x)$ and $\fG_w(\x;\y)$ combinatorially. The row and column weights of pipedreams (resp. marked bumpless pipedreams) are weak compositions that record the number of certain types of tiles in each row and column of the diagrams.
    Let $\PD(w)$ and $\MBPD(w)$ denote the sets of pipedreams and marked bumpless pipedreams labeled by $w$, respectively.
	
	Theorem~\ref{T: PSW} implies that there exists a unique pipedream (resp. marked bumpless pipedream) in $\PD(w)$ (resp. $\MBPD(w)$) whose row weight is $\rajcode(w)$ and column weight is $\rajcode(w^{-1})$, which we call the maximal pipedream (resp. marked bumpless pipedream). Pechenik, Speyer, and Weigandt asked in their paper for an explicit algorithm that constructs the maximal pipedream for any permutation, which was subsequently given by Chou and Yu~\cite{cy24}. In this paper, we provide an explicit algorithm that constructs the maximal marked bumpless pipedream starting with the snow Rothe pipedream $\SRPD(w)$ (see Section~\ref{s: rajcode and snow}).
	
	\begin{thm}
		\label{T: Main Theorem}
        For $w \in S_n$, let $\widehat{\mathrm{D}}(w)$ be the marked bumpless pipedream obtained by applying the algorithm in Section~\ref{s: algorithm} to the snow Rothe pipedream of $w$. Then $\widehat{\mathrm{D}}(w)$ is the unique maximal marked bumpless pipedream of $w$ with row weight $\rajcode(w)$ and column weight $\rajcode(w^{-1})$.
	\end{thm}

     Huang, Shimozono, and Yu~\cite{HSY24} established a canonical bijection between $\PD(w)$ and $\MBPD(w)$ that preserves row weights but not column weights. In fact, there does not exist a bijection between pipedreams and marked bumpless pipedreams that preserves both row and column weights. We remark that composing the constructions in~\cite{cy24} and~\cite{HSY24} does yield a bumpless pipedream with row weight $\rajcode(w)$, but not necessarily with column weight $\rajcode(w^{-1})$ (see Example~\ref{e:hsy}).
	
	The outline of the paper is as follows. In Section~\ref{S: Background}, we introduce the necessary background. In Section~\ref{S: Algorithm}, we describe the algorithm that constructs the maximal marked bumpless pipedream. In Section~\ref{S: Proof of main theorem}, we prove Theorem~\ref{T: Main Theorem}.

	\section*{Acknowledgement}
    This work was supported by the Natural Science Foundation of Tianjin City
 (Grant No. 25JCQNJC00250) and  the Natural Science Foundation of China (Grant No. 12001398). 
 
	We would like to thank Jack Chou, Peter Guo, Zachary Hamaker, and Tianyi Yu for helpful conversations. We thank Jack Chou for helping with earlier drafts and the exposition of this paper. 
	\section{Background}
	\label{S: Background}
	Let $S_n$ be the set of permutations on $[n] = \{1, \dots, n\}$. For an $n \times n$ grid, we label the rows from top to bottom and the columns from left to right, and $(i,j)$ indicates the cell in row $i$ and column $j$. A \defi{diagram} is a subset of the $n \times n$ grid. A \defi{weak composition} \(\alpha\) is a finite sequence of $\Z_{\geqslant 0}$. We write \(\alpha_i\) for its $i\textsuperscript{th}$ entry.

	\subsection{Marked bumpless pipedreams and Grothendieck polynomials}
	\begin{defn}[\cite{lls21,lls23,weigandt21}]
		A \defi{marked bumpless pipedream} is a tiling of the $n \times n$ grid using the tiles $\ptile$, $\jtile$, $\rtile$, $\htile$, $\vtile$, $\btile$, $\mtile$ such that
		\begin{enumerate}
			\item there are $n$ total pipes,
			\item each pipe starts from the bottom edge of the grid, 
			\item each pipe ends at the right edge of the grid.
		\end{enumerate}
		\end{defn}
		Each marked bumpless pipedream is associated with a permutation $w \in S_n$. We first label the pipes by the column they enter from the bottom, and trace the pipes until they exit to the right. We then read off the labeling on the right as the one-line notation of $w$. When we read the permutation \(w\) from the diagram, if a pair of pipes cross each other more than once, we ignore all crossings after the first one. We denote the set of marked bumpless pipedreams associated with $w$ by $\MBPD(w)$. 
        
		For a marked bumpless pipedream $P$, the \defi{row weight} of $P$ is the weak composition $\rwt(P) = (r_1, \dots, r_n)$ where $r_i$ is the number of $\btile$ and $\mtile$ in row $i$ of $P$. Similarly, the \defi{column weight} of $P$ is the weak composition $\cwt(P) = (c_1, \dots, c_n)$ where $c_j$ is the number of $\btile$ and $\mtile$ in column $j$ of $P$. 

	\begin{example}
		The following are two marked bumpless pipedreams $P_1, P_2 \in \MBPD(251634)$. 
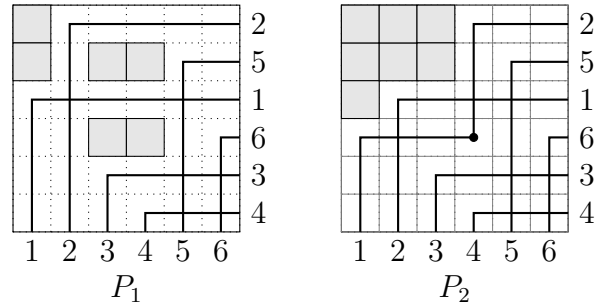
		$$
		\begin{tikzpicture}[scale = 0.5][x=1.5em,y=1.5em,thick,color=blue]
    \def\squarepath{-- +(6,0) -- +(6,6) -- +(0,6) -- cycle}
		\draw (0mm,0mm)\squarepath;
		\draw [step=1,dotted] (0,0) grid (6,6);
			\draw[thick] (0.5,0)--(0.5,3.5)--(6,3.5);
			\draw[thick] (1.5,0)--(1.5,5.5)--(6,5.5);
			\draw[thick] (2.5,0)--(2.5,1.5)--(6,1.5);
			\draw[thick] (3.5,0)--(3.5,0.5)--(6,0.5);
			\draw[thick] (4.5,0)--(4.5,4.5)--(6,4.5);
			\draw[thick] (5.5,0)--(5.5,2.5)--(6,2.5);
			\node[color=black] at (6.5,5.5) {$2$};
			\node[color=black] at (6.5,4.5) {$5$};
			\node[color=black] at (6.5,3.5) {$1$};
			\node[color=black] at (6.5,2.5) {$6$};
			\node[color=black] at (6.5,1.5) {$3$};
			\node[color=black] at (6.5,0.5) {$4$};
			\node[color=black] at (0.5,-0.5) {$1$};
			\node[color=black] at (1.5,-0.5) {$2$};
			\node[color=black] at (2.5,-0.5) {$3$};
			\node[color=black] at (3.5,-0.5) {$4$};
			\node[color=black] at (4.5,-0.5) {$5$};
			\node[color=black] at (5.5,-0.5) {$6$};
    \filldraw [fill=gray!20!white, draw=black!20!black] (0,5) rectangle (1,6);
    \filldraw [fill=gray!20!white, draw=black!20!black] (0,4) rectangle (1,5);
    \filldraw [fill=gray!20!white, draw=black!20!black] (2,4) rectangle (3,5);
    \filldraw [fill=gray!20!white, draw=black!20!black] (2,2) rectangle (3,3);
    \filldraw [fill=gray!20!white, draw=black!20!black] (3,4) rectangle (4,5);
    \filldraw [fill=gray!20!white, draw=black!20!black] (3,2) rectangle (4,3);
    \node[color=black] at (3,-1.5) {$P_1$};
		\end{tikzpicture}
		\quad \quad
		\begin{tikzpicture}[scale = 0.5]
    \def\squarepath{-- +(6,0) -- +(6,6) -- +(0,6) -- cycle}
		\draw (0mm,0mm)\squarepath;
		\draw [step=1,dotted] (0,0) grid (6,6);
			\draw[step=1,gray,ultra thin] (0,0) grid (6,6);
        \def\rectanglepath{[fill=gray!20!white,draw=black!20!black]-- +(1,0) -- +(1,1) -- +(0,1) -- cycle}
		\draw (0,4) \rectanglepath;
        \draw (0,5) \rectanglepath;
        \draw (1,4) \rectanglepath;
        \draw (1,5) \rectanglepath;
        \draw (2,5) \rectanglepath;
        \draw (2,4) \rectanglepath;
        \draw (0,3) \rectanglepath;
			\draw[thick] (0.5,0)--(0.5,2.5)--(3.5,2.5)--(3.5,5.5)--(6,5.5);
			\draw[thick] (1.5,0)--(1.5,3.5)--(6,3.5);
			\draw[thick] (2.5,0)--(2.5,1.5)--(6,1.5);
			\draw[thick] (3.5,0)--(3.5,0.5)--(6,0.5);
			\draw[thick] (4.5,0)--(4.5,4.5)--(6,4.5);
			\draw[thick] (5.5,0)--(5.5,2.5)--(6,2.5);
			\node[color=black] at (6.5,5.5) {$2$};
			\node[color=black] at (6.5,4.5) {$5$};
			\node[color=black] at (6.5,3.5) {$1$};
			\node[color=black] at (6.5,2.5) {$6$};
			\node[color=black] at (6.5,1.5) {$3$};
			\node[color=black] at (6.5,0.5) {$4$};
			\node[color=black] at (0.5,-0.5) {$1$};
			\node[color=black] at (1.5,-0.5) {$2$};
			\node[color=black] at (2.5,-0.5) {$3$};
			\node[color=black] at (3.5,-0.5) {$4$};
			\node[color=black] at (4.5,-0.5) {$5$};
			\node[color=black] at (5.5,-0.5) {$6$};
    \filldraw[black] (3.5,2.5) circle (3pt);
    \node[color=black] at (3,-1.5) {$P_2$};
		\end{tikzpicture}
		$$
		The weights of the marked bumpless pipedreams are 
		\begin{align*}
			\rwt(P_1) = (1,3,0,2,0,0), \quad \rwt(P_2) = (3,3,1,1,0,0),\\
			\cwt(P_1) = (2,0,2,2,0,0), \quad \cwt(P_2) = (3,2,2,1,0,0).
		\end{align*}
	\end{example}
	
	\begin{thm}[\cite{weigandt21}, Theorem 1.1]
		The \defi{Grothendieck polynomials} $\fG_w(\x)$ and \defi{double Grothendieck polynomials} $\fG_w(\x; \y)$  are given by the following weighted sums: 
		\begin{align*}
			\fG_w(\x) &:=\sum_{P \in \MBPD(w)}  (-1)^{\#\{ \btile\, ,\, \mtile\}-\ell(w)}\prod_{(i,j)\in  \btile\, , \, \mtile}x_i,\\
			\fG_w(\x;\y)&: =\sum_{P \in \MBPD(w)}  (-1)^{\#\{ \btile\, , \, \mtile\}-\ell(w)}\prod_{(i,j)\in  \btile\, , \, \mtile}(x_i+y_j-x_iy_j).
		\end{align*}
	\end{thm}
		
	\begin{example}
		\label{e: Grothendieck}
		For $w = 1423$, there are five marked bumpless pipedreams in $\MBPD(w)$.
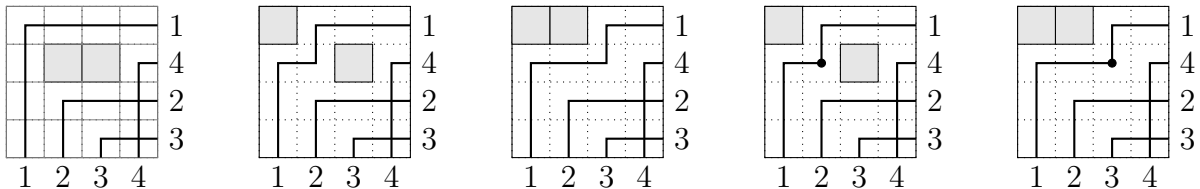
		$$
		\begin{tikzpicture}[scale=0.5]
    \def\squarepath{-- +(4,0) -- +(4,4) -- +(0,4) -- cycle}
		\draw (0mm,0mm)\squarepath;
		\draw [step=1,dotted] (0,0) grid (4,4);
        \def\rectanglepath{[fill=gray!20!white,draw=black!20!black]-- +(1,0) -- +(1,1) -- +(0,1) -- cycle}
		\draw (1,2) \rectanglepath;
        \draw (2,2) \rectanglepath;
        
			\draw[step=1,gray,ultra thin] (0,0) grid (4,4);
			\draw[thick] (0.5,0)--(0.5,3.5)--(4,3.5);
			\draw[thick] (1.5,0)--(1.5,1.5)--(4,1.5);
			\draw[thick] (2.5,0)--(2.5,0.5)--(4,0.5);
			\draw[thick] (3.5,0)--(3.5,2.5)--(4,2.5);
			\node[color=black] at (4.5,3.5) {$1$};
			\node[color=black] at (4.5,2.5) {$4$};
			\node[color=black] at (4.5,1.5) {$2$};
			\node[color=black] at (4.5,0.5) {$3$};
			\node[color=black] at (0.5,-0.5) {$1$};
			\node[color=black] at (1.5,-0.5) {$2$};
			\node[color=black] at (2.5,-0.5) {$3$};
			\node[color=black] at (3.5,-0.5) {$4$};
		\end{tikzpicture}
		\quad \quad 
		\begin{tikzpicture}[scale=0.5]
    \def\squarepath{-- +(4,0) -- +(4,4) -- +(0,4) -- cycle}
		\draw (0mm,0mm)\squarepath;
		\draw [step=1,dotted] (0,0) grid (4,4);
		\def\rectanglepath{[fill=gray!20!white,draw=black!20!black]-- +(1,0) -- +(1,1) -- +(0,1) -- cycle}
        \draw (2,2) \rectanglepath;
        \draw (0,3) \rectanglepath;
        
			\draw[thick] (0.5,0)--(0.5,2.5)--(1.5,2.5)--(1.5,3.5)--(4,3.5);
			\draw[thick] (1.5,0)--(1.5,1.5)--(4,1.5);
			\draw[thick] (2.5,0)--(2.5,0.5)--(4,0.5);
			\draw[thick] (3.5,0)--(3.5,2.5)--(4,2.5);
			\node[color=black] at (4.5,3.5) {$1$};
			\node[color=black] at (4.5,2.5) {$4$};
			\node[color=black] at (4.5,1.5) {$2$};
			\node[color=black] at (4.5,0.5) {$3$};
			\node[color=black] at (0.5,-0.5) {$1$};
			\node[color=black] at (1.5,-0.5) {$2$};
			\node[color=black] at (2.5,-0.5) {$3$};
			\node[color=black] at (3.5,-0.5) {$4$};
       
		\end{tikzpicture}
\quad \quad 
		\begin{tikzpicture}[scale=0.5]
    \def\squarepath{-- +(4,0) -- +(4,4) -- +(0,4) -- cycle}
		\draw (0mm,0mm)\squarepath;
		\draw [step=1,dotted] (0,0) grid (4,4);
        \def\rectanglepath{[fill=gray!20!white,draw=black!20!black]-- +(1,0) -- +(1,1) -- +(0,1) -- cycle}
		\draw (1,3) \rectanglepath;
        \draw (0,3) \rectanglepath;
			\draw[thick] (0.5,0)--(0.5,2.5)--(2.5,2.5)--(2.5,3.5)--(4,3.5);
			\draw[thick] (1.5,0)--(1.5,1.5)--(4,1.5);
			\draw[thick] (2.5,0)--(2.5,0.5)--(4,0.5);
			\draw[thick] (3.5,0)--(3.5,2.5)--(4,2.5);
			\node[color=black] at (4.5,3.5) {$1$};
			\node[color=black] at (4.5,2.5) {$4$};
			\node[color=black] at (4.5,1.5) {$2$};
			\node[color=black] at (4.5,0.5) {$3$};
			\node[color=black] at (0.5,-0.5) {$1$};
			\node[color=black] at (1.5,-0.5) {$2$};
			\node[color=black] at (2.5,-0.5) {$3$};
			\node[color=black] at (3.5,-0.5) {$4$};
          
		\end{tikzpicture}
		\quad \quad 
        \begin{tikzpicture}[scale=0.5]
    \def\squarepath{-- +(4,0) -- +(4,4) -- +(0,4) -- cycle}
		\draw (0mm,0mm)\squarepath;
		\draw [step=1,dotted] (0,0) grid (4,4);
		\def\rectanglepath{[fill=gray!20!white,draw=black!20!black]-- +(1,0) -- +(1,1) -- +(0,1) -- cycle}
        \draw (2,2) \rectanglepath;
        \draw (0,3) \rectanglepath;
        
			\draw[thick] (0.5,0)--(0.5,2.5)--(1.5,2.5)--(1.5,3.5)--(4,3.5);
			\draw[thick] (1.5,0)--(1.5,1.5)--(4,1.5);
			\draw[thick] (2.5,0)--(2.5,0.5)--(4,0.5);
			\draw[thick] (3.5,0)--(3.5,2.5)--(4,2.5);
			\node[color=black] at (4.5,3.5) {$1$};
			\node[color=black] at (4.5,2.5) {$4$};
			\node[color=black] at (4.5,1.5) {$2$};
			\node[color=black] at (4.5,0.5) {$3$};
			\node[color=black] at (0.5,-0.5) {$1$};
			\node[color=black] at (1.5,-0.5) {$2$};
			\node[color=black] at (2.5,-0.5) {$3$};
			\node[color=black] at (3.5,-0.5) {$4$};
            \filldraw[black] (1.5,2.5) circle (3pt);
		\end{tikzpicture}
		\quad \quad 
		\begin{tikzpicture}[scale=0.5]
    \def\squarepath{-- +(4,0) -- +(4,4) -- +(0,4) -- cycle}
		\draw (0mm,0mm)\squarepath;
		\draw [step=1,dotted] (0,0) grid (4,4);
        \def\rectanglepath{[fill=gray!20!white,draw=black!20!black]-- +(1,0) -- +(1,1) -- +(0,1) -- cycle}
		\draw (1,3) \rectanglepath;
        \draw (0,3) \rectanglepath;
			\draw[thick] (0.5,0)--(0.5,2.5)--(2.5,2.5)--(2.5,3.5)--(4,3.5);
			\draw[thick] (1.5,0)--(1.5,1.5)--(4,1.5);
			\draw[thick] (2.5,0)--(2.5,0.5)--(4,0.5);
			\draw[thick] (3.5,0)--(3.5,2.5)--(4,2.5);
			\node[color=black] at (4.5,3.5) {$1$};
			\node[color=black] at (4.5,2.5) {$4$};
			\node[color=black] at (4.5,1.5) {$2$};
			\node[color=black] at (4.5,0.5) {$3$};
			\node[color=black] at (0.5,-0.5) {$1$};
			\node[color=black] at (1.5,-0.5) {$2$};
			\node[color=black] at (2.5,-0.5) {$3$};
			\node[color=black] at (3.5,-0.5) {$4$};
            \filldraw[black] (2.5,2.5) circle (3pt);
		\end{tikzpicture}
		$$
		Therefore, the Grothendieck polynomial and double Grothendieck polynomial of $w$ are
		\begin{align*}
			\fG_w(\x) = &\, x_2^2 + x_1x_2 + x_1^2 - x_1x_2^2 - x_1^2x_2,\\
			\fG_w(\x; \y) = &\,(x_2 + y_2 - x_2y_2)(x_2 + y_3 - x_2y_3) \\
			+ &(x_1 + y_1 - x_1y_1)(x_2 + y_3 - x_2y_3)\\
			+ &(x_1 + y_1 - x_1y_1)(x_1 + y_2 - x_1y_2)\\
            - &(x_1 + y_1 - x_1y_1)(x_2 + y_2 - x_2y_2)(x_2 + y_3 - x_2y_3)\\
            - &(x_1 + y_1 - x_1y_1)(x_1 + y_2 - x_1y_2)(x_2 + y_3 - x_2y_3).
		\end{align*}
	\end{example}
	Let $\widehat{\MBPD}(w)$ denote the set of marked bumpless pipedreams in $\MBPD(w)$ with maximal number of $\btile$ and $\mtile$. We obtain the \defi{(double) Castelnuovo–Mumford polynomials} by taking the highest degree terms in each product in the formula above from marked bumpless pipedreams in $\widehat{\MBPD}(w)$.
	\begin{align*}
		\fCM_w(\x) &= \sum_{P \in \widehat{\MBPD}(w)} \prod_{(i,j) = \btile, \mtile} x_i =  \sum_{P \in \widehat{\MBPD}(w)} x^{\rwt(P)}\\
		\fCM_w(\x; \y) &= \sum_{P \in \widehat{\MBPD}(w)} \prod_{(i,j) = \btile, \mtile} x_iy_j =  \sum_{P \in \widehat{\MBPD}(w)} x^{\rwt(P)}y^{\cwt(P)}
	\end{align*}
	
	\begin{example}
		Continuing Example \ref{e: Grothendieck}, let $w  = 1423$. There are two marked bumpless pipedreams in $\widehat{\MBPD}(w)$.
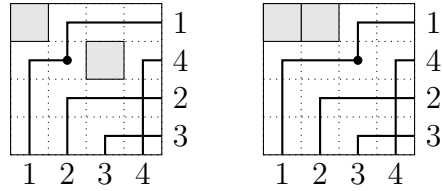
		$$
		\begin{tikzpicture}[scale=0.5]
    \def\squarepath{-- +(4,0) -- +(4,4) -- +(0,4) -- cycle}
		\draw (0mm,0mm)\squarepath;
		\draw [step=1,dotted] (0,0) grid (4,4);
        \def\rectanglepath{[fill=gray!20!white,draw=black!20!black]-- +(1,0) -- +(1,1) -- +(0,1) -- cycle}
        \draw (2,2) \rectanglepath;
        \draw (0,3) \rectanglepath;
			\draw[thick] (0.5,0)--(0.5,2.5)--(1.5,2.5)--(1.5,3.5)--(4,3.5);
			\draw[thick] (1.5,0)--(1.5,1.5)--(4,1.5);
			\draw[thick] (2.5,0)--(2.5,0.5)--(4,0.5);
			\draw[thick] (3.5,0)--(3.5,2.5)--(4,2.5);
			\node[color=black] at (4.5,3.5) {$1$};
			\node[color=black] at (4.5,2.5) {$4$};
			\node[color=black] at (4.5,1.5) {$2$};
			\node[color=black] at (4.5,0.5) {$3$};
			\node[color=black] at (0.5,-0.5) {$1$};
			\node[color=black] at (1.5,-0.5) {$2$};
			\node[color=black] at (2.5,-0.5) {$3$};
			\node[color=black] at (3.5,-0.5) {$4$};
            \filldraw[black] (1.5,2.5) circle (3pt);
		\end{tikzpicture}
		\quad \quad 
		\begin{tikzpicture}[scale=0.5]
    \def\squarepath{-- +(4,0) -- +(4,4) -- +(0,4) -- cycle}
		\draw (0mm,0mm)\squarepath;
		\draw [step=1,dotted] (0,0) grid (4,4);
        \def\rectanglepath{[fill=gray!20!white,draw=black!20!black]-- +(1,0) -- +(1,1) -- +(0,1) -- cycle}
		\draw (1,3) \rectanglepath;
        \draw (0,3) \rectanglepath;
			\draw[thick] (0.5,0)--(0.5,2.5)--(2.5,2.5)--(2.5,3.5)--(4,3.5);
			\draw[thick] (1.5,0)--(1.5,1.5)--(4,1.5);
			\draw[thick] (2.5,0)--(2.5,0.5)--(4,0.5);
			\draw[thick] (3.5,0)--(3.5,2.5)--(4,2.5);
			\node[color=black] at (4.5,3.5) {$1$};
			\node[color=black] at (4.5,2.5) {$4$};
			\node[color=black] at (4.5,1.5) {$2$};
			\node[color=black] at (4.5,0.5) {$3$};
			\node[color=black] at (0.5,-0.5) {$1$};
			\node[color=black] at (1.5,-0.5) {$2$};
			\node[color=black] at (2.5,-0.5) {$3$};
			\node[color=black] at (3.5,-0.5) {$4$};
            \filldraw[black] (2.5,2.5) circle (3pt);
		\end{tikzpicture}
		$$
		Therefore, the Castelnuovo–Mumford polynomial and double Castelnuovo–Mumford polynomial of $w$ are
		\begin{align*}
			\fCM_w(\x) = \,x_1x_2^2 + x_1^2x_2 \quad \text{ and } \quad \fCM_w(\x;\y) = \,x_1x_2^2y_1y_2y_3 + x_1^2x_2y_1y_2y_3.
		\end{align*}
	\end{example}
	\begin{rem}
		Since the signs of monomials in a (double) Grothendieck polynomial alternate with degree, it could be the case that its highest degree homogeneous component consists of monomials with all negative coefficients (see Example~\ref{e: Grothendieck}). For the purposes of this paper, only the supports and weights of the top-degree terms are relevant. We therefore adopt the monomial-positive convention for the (double) Castelnuovo--Mumford polynomials.
	\end{rem}
	
	\subsection{Rajchgot code and snow diagrams}
	\label{s: rajcode and snow}
    Pechenik, Speyer, and Weigandt~\cite{psw24} defined the Rajchgot code, denoted by \(\rajcode(\cdot)\), to describe the leading monomial of \(\fCM_w(x)\) under reverse lexicographic order.
    For $w \in S_n$, they computed $\rajcode(w)$ by considering longest increasing subsequences in the one-line notation of $w$. In this paper, we use another combinatorial formula introduced by Pan and Yu as our definition of $\rajcode(w)$.
	
    The \defi{Rothe diagram} of a permutation $w$ is the diagram $\Rothe(w) = \{(i,w(j)): i < j, w(i) > w(j)\}.$ For every $w\in S_n$, there exists a unique marked bumpless pipedream whose $\btile$ are exactly $\Rothe(w)$. We call this the \defi{Rothe pipedream} of $w$ and denote it by $\RPD(w)$. It is also the unique marked bumpless pipedream that does not use any $\jtile$ or $\mtile$.
	
	\begin{example}
		For $w = 251634$, the diagrams $\Rothe(w)$ and $\RPD(w)$ are shown below.
		$$
		
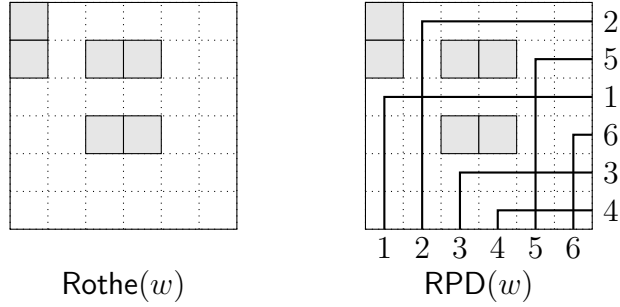
\begin{tikzpicture}[scale = 0.5][x=1.5em,y=1.5em,thick,color=blue]
    \def\squarepath{-- +(6,0) -- +(6,6) -- +(0,6) -- cycle}
		\draw (0mm,0mm)\squarepath;
		\draw [step=1,dotted] (0,0) grid (6,6);
    \filldraw [fill=gray!20!white, draw=black!20!black] (0,5) rectangle (1,6);
    \filldraw [fill=gray!20!white, draw=black!20!black] (0,4) rectangle (1,5);
    \filldraw [fill=gray!20!white, draw=black!20!black] (2,4) rectangle (3,5);
    \filldraw [fill=gray!20!white, draw=black!20!black] (2,2) rectangle (3,3);
    \filldraw [fill=gray!20!white, draw=black!20!black] (3,4) rectangle (4,5);
    \filldraw [fill=gray!20!white, draw=black!20!black] (3,2) rectangle (4,3);
    \node[color=black] at (0.5,-0.5) {$\quad$};
    \node[color=black] at (3,-1.5) {$\Rothe(w)$};
\end{tikzpicture}
		\quad \quad \quad \quad 
		\begin{tikzpicture}[scale=0.5]
    \def\squarepath{-- +(6,0) -- +(6,6) -- +(0,6) -- cycle}
		\draw (0mm,0mm)\squarepath;
		\draw [step=1,dotted] (0,0) grid (6,6);
        \filldraw [fill=gray!20!white, draw=black!20!black] (0,5) rectangle (1,6);
    \filldraw [fill=gray!20!white, draw=black!20!black] (0,4) rectangle (1,5);
    \filldraw [fill=gray!20!white, draw=black!20!black] (2,4) rectangle (3,5);
    \filldraw [fill=gray!20!white, draw=black!20!black] (2,2) rectangle (3,3);
    \filldraw [fill=gray!20!white, draw=black!20!black] (3,4) rectangle (4,5);
    \filldraw [fill=gray!20!white, draw=black!20!black] (3,2) rectangle (4,3);
    
    \draw[thick] (0.5,0)--(0.5,3.5)--(6,3.5);
    \draw[thick] (1.5,0)--(1.5,5.5)--(6,5.5);
    \draw[thick] (2.5,0)--(2.5,1.5)--(6,1.5);
    \draw[thick] (3.5,0)--(3.5,0.5)--(6,0.5);
    \draw[thick] (4.5,0)--(4.5,4.5)--(6,4.5);
    \draw[thick] (5.5,0)--(5.5,2.5)--(6,2.5);
    
    \node[color=black] at (6.5,5.5) {$2$};
    \node[color=black] at (6.5,4.5) {$5$};
    \node[color=black] at (6.5,3.5) {$1$};
    \node[color=black] at (6.5,2.5) {$6$};
    \node[color=black] at (6.5,1.5) {$3$};
    \node[color=black] at (6.5,0.5) {$4$};
    \node[color=black] at (0.5,-0.5) {$1$};
    \node[color=black] at (1.5,-0.5) {$2$};
    \node[color=black] at (2.5,-0.5) {$3$};
    \node[color=black] at (3.5,-0.5) {$4$};
    \node[color=black] at (4.5,-0.5) {$5$};
    \node[color=black] at (5.5,-0.5) {$6$};
    \node[color=black] at (3,-1.5) {$\RPD(w)$};
\end{tikzpicture}
		$$
	\end{example}
    
	For a diagram $D$, we define the \defi{dark clouds} of $D$ to be the subset $\dark(D) \subseteq D$ constructed as follows: Scan through $D$ from bottom to top. For each row $r$, if there exists $(r,c) \in D$ such that currently there are no cells in column $c$ of $\dark(D)$, we find the largest such $c$ and put $(r,c)$ in $\dark(D)$.

    \begin{example}
        
        For $w=251634$, the diagrams $\Rothe(w)$ and $\dark(\Rothe(w))$ are shown below.
        $$
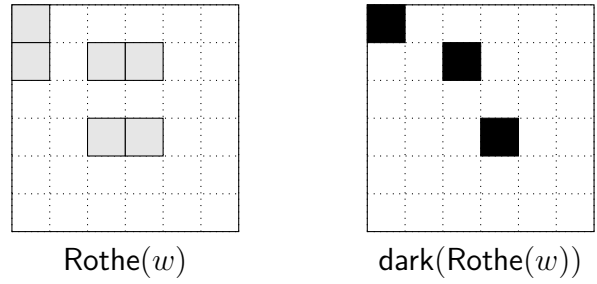
\begin{tikzpicture}[scale = 0.5][x=1.5em,y=1.5em,thick,color=blue]
    \def\squarepath{-- +(6,0) -- +(6,6) -- +(0,6) -- cycle}
		\draw (0mm,0mm)\squarepath;
		\draw [step=1,dotted] (0,0) grid (6,6);
    \filldraw [fill=gray!20!white, draw=black!20!black] (0,5) rectangle (1,6);
    \filldraw [fill=gray!20!white, draw=black!20!black] (0,4) rectangle (1,5);
    \filldraw [fill=gray!20!white, draw=black!20!black] (2,4) rectangle (3,5);
    \filldraw [fill=gray!20!white, draw=black!20!black] (2,2) rectangle (3,3);
    \filldraw [fill=gray!20!white, draw=black!20!black] (3,4) rectangle (4,5);
    \filldraw [fill=gray!20!white, draw=black!20!black] (3,2) rectangle (4,3);
    \node[color=black] at (3,-0.8) {$\Rothe(w)$};
\end{tikzpicture}
\quad \quad \quad \quad 
\begin{tikzpicture}[scale = 0.5][x=1.5em,y=1.5em,thick,color=blue]
    \def\squarepath{-- +(6,0) -- +(6,6) -- +(0,6) -- cycle}
		\draw (0mm,0mm)\squarepath;
		\draw [step=1,dotted] (0,0) grid (6,6);
    \filldraw [black] (0,5) rectangle (1,6);
    \filldraw [black] (2,4) rectangle (3,5);
    \filldraw [black] (3,2) rectangle (4,3);
    \node[color=black] at (3,-0.8) {$\dark(\Rothe(w))$};
\end{tikzpicture}
        $$
    \end{example}
		
	The \defi{snow diagram} of $D$, denoted as $\Snow(D)$, is the diagram obtained by filling in all empty cells above $\dark(D)$ in $D$. The \defi{left snow diagram} of $D$,  denoted as $\overleftarrow{\Snow}(D)$, is the diagram obtained by filling in all empty cells to the left of $\dark(D)$ in $D$. We call these additional cells \defi{snow cells}. For $w \in S_n$, we further define $\Snow(w) := \Snow(\Rothe(w))$ and $\overleftarrow{\Snow}(w) := \overleftarrow{\Snow}(\Rothe(w))$. For each fixed pipe, the number of {\large{$*$}} in $\Snow(w)$ is equal to that in $\overleftarrow{\Snow}(w)$. Define $\SRPD(w)$ as the diagram obtained by overlaying the $\Snow(w)$ and the $\RPD(w)$, i.e., a diagram that contains both pipes and {\large{$*$}}. 
	
	\begin{example}
		\label{e: snow}
Continuing the previous example, we illustrate $\Snow(w)$ , $\overleftarrow{\Snow}(w)$ and $\SRPD(w)$ for $w = 251634$. For clarity, dark clouds are represented by black cells, while snow cells in the (left) snow diagrams are marked by stars.
		$$
		
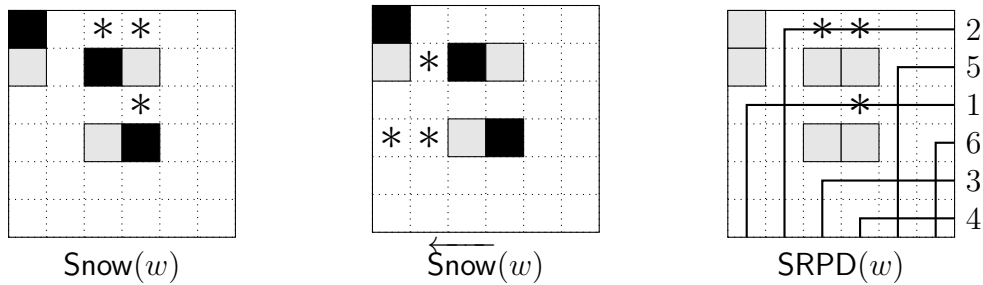
\begin{tikzpicture}[scale = 0.5][x=1.5em,y=1.5em,thick,color=blue]
    
    \def\squarepath{-- +(6,0) -- +(6,6) -- +(0,6) -- cycle}
		\draw (0mm,0mm)\squarepath;
		\draw [step=1,dotted] (0,0) grid (6,6);
    
    \filldraw [black] (0,5) rectangle (1,6);  
    \filldraw [fill=gray!20!white, draw=black!20!black]  (0,4) rectangle (1,5);   
    \filldraw [black] (2,4) rectangle (3,5);   
    \filldraw [fill=gray!20!white, draw=black!20!black]  (2,2) rectangle (3,3);   
    \filldraw [fill=gray!20!white, draw=black!20!black]  (3,4) rectangle (4,5);  
    \filldraw [black] (3,2) rectangle (4,3);   
    
    \node[black, scale=1.5] at (2.5,5.5) {$*$};
    \node[black, scale=1.5] at (3.5,5.5) {$*$};
    \node[black, scale=1.5] at (3.5,3.5) {$*$};
    \node[color=black] at (3,-0.8) {$\Snow(w)$};
\end{tikzpicture}
		\quad\quad\quad\quad
		\begin{tikzpicture}[scale = 0.5][x=1.5em,y=1.5em,thick,color=blue]
    \def\squarepath{-- +(6,0) -- +(6,6) -- +(0,6) -- cycle}
		\draw (0mm,0mm)\squarepath;
		\draw [step=1,dotted] (0,0) grid (6,6);
    
    \filldraw [black] (0,5) rectangle (1,6);   
    \filldraw [fill=gray!20!white, draw=black!20!black]  (0,4) rectangle (1,5);   
    \filldraw [black] (2,4) rectangle (3,5);  
    \filldraw [fill=gray!20!white, draw=black!20!black]  (2,2) rectangle (3,3);   
    \filldraw [fill=gray!20!white, draw=black!20!black]  (3,4) rectangle (4,5);   
    \filldraw [black] (3,2) rectangle (4,3);   
    
    \node[black, scale=1.5] at (0.5,2.5) {$*$};
    \node[black, scale=1.5] at (1.5,2.5) {$*$};
    \node[black, scale=1.5] at (1.5,4.5) {$*$};
    \node[color=black] at (3,-0.8) {$\overleftarrow{\Snow}(w)$};
\end{tikzpicture}
\quad\quad\quad\quad
\begin{tikzpicture}[scale = 0.5][x=1.5em,y=1.5em,thick,color=blue]
    
    \def\squarepath{-- +(6,0) -- +(6,6) -- +(0,6) -- cycle}
		\draw (0mm,0mm)\squarepath;
		\draw [step=1,dotted] (0,0) grid (6,6);
    
    \filldraw [fill=gray!20!white, draw=black!20!black] (0,5) rectangle (1,6);  
    \filldraw [fill=gray!20!white, draw=black!20!black]  (0,4) rectangle (1,5);   
    \filldraw [fill=gray!20!white, draw=black!20!black] (2,4) rectangle (3,5);   
    \filldraw [fill=gray!20!white, draw=black!20!black]  (2,2) rectangle (3,3);   
    \filldraw [fill=gray!20!white, draw=black!20!black]  (3,4) rectangle (4,5);  
    \filldraw [fill=gray!20!white, draw=black!20!black] (3,2) rectangle (4,3);   
    \draw[thick] (0.5,0)--(0.5,3.5)--(6,3.5);
    \draw[thick] (1.5,0)--(1.5,5.5)--(6,5.5);
    \draw[thick] (2.5,0)--(2.5,1.5)--(6,1.5);
    \draw[thick] (3.5,0)--(3.5,0.5)--(6,0.5);
    \draw[thick] (4.5,0)--(4.5,4.5)--(6,4.5);
    \draw[thick] (5.5,0)--(5.5,2.5)--(6,2.5);
    \node[black, scale=1.5] at (2.5,5.5) {$*$};
    \node[black, scale=1.5] at (3.5,5.5) {$*$};
    \node[black, scale=1.5] at (3.5,3.5) {$*$};
    \node[color=black] at (6.5,5.5) {$2$};
    \node[color=black] at (6.5,4.5) {$5$};
    \node[color=black] at (6.5,3.5) {$1$};
    \node[color=black] at (6.5,2.5) {$6$};
    \node[color=black] at (6.5,1.5) {$3$};
    \node[color=black] at (6.5,0.5) {$4$};
    \node[color=black] at (3,-0.8) {$\SRPD(w)$};
\end{tikzpicture}
		$$
	\end{example}
	
	\begin{defn}
		For a diagram $D$, the \defi{row weight} of $D$ is the weak composition $\rw(D)$ where $\rw(D)_i$ is the number of cells in row $i$ of $D$. Similarly, the \defi{column weight} of $D$ is the weak composition $\cw(D)$ where $\cw(D)_i$ is the number of cells in column $i$ of $D$.
	\end{defn}
    
    For each $w\in S_n$, Pechenik, Speyer, and Weigandt originally defined the weak composition $\rajcode(w)$ via increasing subsequences of $w$.  In  this  paper,  we adopt the diagrammatic definition of Pan and Yu \cite{py24}, which uses snow diagrams.
    
	\begin{defn}[\cite{py24}*{Thm.~5.6}]\label{def}
		For $w \in S_n$, $\rajcode(w) := \rw(\Snow(w))$, and $\rajcode(w^{-1}) := \cw(\overleftarrow{\Snow}(w))$.
	\end{defn}
	
	\begin{example}
		For $w = 251634$, we compute $\rajcode(w)$ and $\rajcode(w^{-1})$ using the (left) snow diagrams in Example~\ref{e: snow}. By counting the number of cells in each row (resp. column) of the snow (resp. left snow) diagram, we get that 
		$\rajcode(w) = (3,3,1,2,0,0)$ and $\rajcode(w^{-1}) = (3,2,2,2,0,0)$. Therefore, by Theorem~\ref{T: PSW}, the leading monomial of $\fCM_{251634}(\x;\y)$ under reverse lexicographic order is $x_1^3x_2^3x_3^1x_4^2y_1^3y_2^2y_3^2y_4^2$.
	\end{example}

	\section{Algorithm}
	\label{S: Algorithm}
	In this section, we describe our algorithm that transforms the snow Rothe pipedream $\SRPD(w)$ into the unique maximal marked bumpless pipedream of $w$. Since we are looking for the ``maximal" diagram, we mark all $\jtile$ as $\mtile$. 
	
	The algorithm begins with the snow Rothe pipedream \(\SRPD(w)\). Roughly speaking, the goal of the algorithm is to iteratively ``release" the snow cells by modifying pipes so that these cells are converted into horizontal segments, while preserving the permutation  (see Example~\ref{e: algorithm}). Let $\widehat{\mathrm{D}}(w)$ denote the marked bumpless pipedream we get from applying the algorithm to $\SRPD(w)$. 

    \subsection{Notations}
	We first introduce some terminologies used in the algorithm. For a pipe $p$, its \defi{starting row} is $w^{-1}(p)$ and its \defi{terminating column} is $p$. We say that a $\htile$ of a pipe is a \defi{long line} if it does not lie in the starting row of that pipe. 
    
	\subsection{Algorithm}
    \label{s: algorithm}
    The algorithm consists of two main steps, which we call ``Drooping" and ``Undrooping". 
	\subsubsection{Drooping}	\label{Drooping}	
	Our algorithm iterates on the pipes in reverse order of their starting row, so the algorithm begins with pipe $w(n)$. In each step, if pipe $p$ contains no {\large{$*$}},  the pipe remains unchanged.
	
	If pipe $p$ contains $t$ {\large{$*$}}s, let $(i,j)$ be the $t^{th}$ $\htile$ counting from the $\rtile$ of pipe $p$. We redraw pipe $p$ after replacing $(i,j)$ with a $\rtile$ and continue tracing pipe $p$ by the following local rules:
	\begin{enumerate}
		\item If pipe $p$ enters column $p$, we draw a $\rtile$ followed by $\vtile$s until pipe $p$ exits the entire diagram from the bottom.
		\item If pipe $p$ enters a $\vtile$ from the right, we replace it with a $\ptile$ so pipe $p$ exits from the left. 
		\item If pipe $p$ enters a $\htile$ from the top, we replace it with a $\ptile$ so pipe $p$ exits from the bottom. 
		\item If pipe $p$ enters a $\rtile$ from the top, we replace it with a $\ptile$ so pipe $p$ exits from the left.
		\item If pipe $p$ enters a $\mtile$ from the right, we replace it with a $\ptile$ so pipe $p$ exits from the bottom.
		\item If pipe $p$ enters a $\btile$ from the top, we replace it with a $\mtile$ so pipe $p$ exits from the left.
		\item If pipe $p$ enters a $\btile$ from the right, we replace it with a $\rtile$ so pipe $p$ exits from the bottom.
	\end{enumerate}
   For the processes in (2) and (3), we say that pipe \(p\) ``passes through” the existing pipe in the tile (without changing direction); the resulting \(\ptile\) is termed a \defi{p-cross}. For the processes in (4) and (5), we say that pipe \(p\) ``contacts” the existing pipe in the tile (thereby changing direction); the resulting \(\ptile\) is termed a \defi{c-cross}. 
	
	\subsubsection{Undrooping}	\label{Undrooping}	
	After redrawing each pipe, we check for long lines in all the pipes. If any long line exists, consider the topmost long line of some pipe $q$. Denote by $(a_1, b_0)$ the position where the long line first appears.
	
	\begin{itemize}
		
		\item [{\bf Step 1.}]
		Traverse along pipe $q$ and locate all instances of $\mtile$ (including $\ptile$). If $r$ such instances exist, label them from bottom to top as $(a_1,b_1), (a_2,b_2),\ldots,(a_r,b_r)$. More specifically, denote $a_{r+1}=w^{-1}(q)$ (see Figure \ref{ab}).
		
		\begin{figure}[h]
			\begin{center}
				\begin{tikzpicture}[scale = 1]
					
					\draw [step=4mm,dotted] (-4mm,0mm) grid (32mm,32mm);
					
					\draw[thick](-2mm,0mm)--(-2mm,6mm)--(2mm,6mm)--(2mm,10mm)--(10mm,10mm)--(10mm,14mm)--(14mm,14mm)--(14mm,18mm)--(18mm,18mm)--(18mm,22mm)--(22mm,22mm)--(22mm,26mm)--(32mm,26mm);

					\node[scale=0.7] at (-6mm,26mm) {$a_5$};
					\node[scale=0.7] at (-6mm,22mm) {$a_4$};
					\node[scale=0.7] at (-6mm,18mm) {$a_3$};
					\node[scale=0.7] at (-6mm,14mm) {$a_2$};
					\node[scale=0.7] at (-6mm,10mm) {$a_1$};
					
					\node[scale=0.7] at (6mm,-2mm) {$b_0$};
					\node[scale=0.7] at (10mm,-2mm) {$b_1$};
					\node[scale=0.7] at (14mm,-2mm) {$b_2$};
					\node[scale=0.7] at (18mm,-2mm) {$b_3$};
					\node[scale=0.7] at (22mm,-2mm) {$b_4$};
					
				\end{tikzpicture}
				\caption{Marking method for undrooping operations.}\label{ab}
			\end{center}
			
		\end{figure}
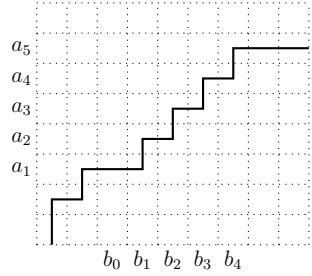
		
		\item [{\bf Step 2.}]
		We construct an operation rectangle whose SE and NW corners are $(a_i, b_i)$ and $(a_{i+1}, b_{i-1})$ for $ i= 1,2,\dots,r$.
		The \defi{mini-undroop} (cf. \cite{HP25}) repositions pipe $q$ from the SE corner of the rectangle to its NW corner, i.e., from $(a_i, b_i)$ to $(a_{i+1}, b_{i-1})$. Figure \ref{rectg} shows this configuration, with all pipes other than $q$ omitted. The resulting diagram after such operations remains a standard marked bumpless pipedream (see Proposition \ref{standard}).

		\begin{figure}[h]
			\begin{center}
				\begin{tabular}{cc}
					\begin{tikzpicture}[scale = 1]
						
						\def\rectanglepath{-- +(16mm,0mm) -- +(16mm,12mm) -- +(0mm,12mm) -- cycle}
						\def\squarepath{[black]-- +(4mm,0mm) -- +(4mm,4mm) -- +(0mm,4mm) -- cycle}
						
						\draw [red](0mm,8mm)\rectanglepath;
						
						%-------
						\draw [step=4mm,dotted] (-4mm,-4mm) grid (24mm,24mm);
						
						\draw[thick](2mm,0mm)--(2mm,10mm)--(14mm,10mm)--(14mm,18mm)--(20mm,18mm);
						
						\draw [->,red](14mm,10mm)--(2mm,18mm);
						
						\draw (0mm,16mm)\squarepath;
						\draw (12mm,8mm)\squarepath;
						
						\node[scale=0.7] at (0mm,22mm) {$(a_{i+1},b_{i-1})$};
						\node[scale=0.7] at (20mm,6mm) {$(a_{i},b_{i})$};

					\end{tikzpicture}&
					
					\begin{tikzpicture}[scale = 1]
						
						\def\rectanglepath{-- +(16mm,0mm) -- +(16mm,12mm) -- +(0mm,12mm) -- cycle}
						\def\squarepath{[black]-- +(4mm,0mm) -- +(4mm,4mm) -- +(0mm,4mm) -- cycle}
						
						\draw [red](0mm,8mm)\rectanglepath;
						
						%-------
						\draw [step=4mm,dotted] (-4mm,-4mm) grid (24mm,24mm);
						
						\draw[thick](2mm,0mm)--(2mm,10mm)--(2mm,18mm)--(14mm,18mm)--(20mm,18mm);
						
						\draw (0mm,16mm)\squarepath;
						\draw (12mm,8mm)\squarepath;
						
						\node[scale=0.7] at (0mm,22mm) {$(a_{i+1},b_{i-1})$};
						\node[scale=0.7] at (20mm,6mm) {$(a_{i},b_{i})$};

					\end{tikzpicture}

				\end{tabular}
				\caption{The mini-undroop in the operation rectangle.}\label{rectg}
				
			\end{center}
			
		\end{figure}
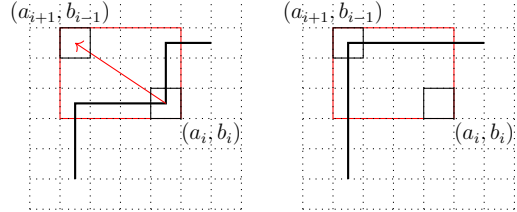
		
		\item[{\bf Step 3.}]
		We apply mini-undroop to the $\mtile$ (including $\ptile$) that lies to the right of this long line. 
        We repeatedly apply mini-undroops until pipe \(q\) has no long lines. Since mini-undroop strictly decreases the row index of $\htile$s, the sequence must terminate after finitely many steps, see Figure \ref{undroop}. We then continue ``Undrooping" process until the entire diagram contains no long lines. Only after all long lines are resolved do we proceed with the ``Drooping" algorithm to redraw the next pipe. 
        
        \begin{figure}[h]
			\begin{center}
				\begin{tabular}{cccc}
					\begin{tikzpicture}[scale = 1]
						\def\rectanglepath{[red]-- +(8mm,0mm) -- +(8mm,8mm) -- +(0mm,8mm) -- cycle}

						%------
						\draw [step=4mm,dotted] (-4mm,0mm) grid (24mm,24mm);
						\draw [ultra thick][black](4mm,6mm)--(8mm,6mm);
						
						\draw[thick](2mm,0mm)--(2mm,6mm)--(10mm,6mm)--(10mm,10mm)--(14mm,10mm)--(14mm,14mm)--(18mm,14mm)--(18mm,18mm)--(24mm,18mm);
						\draw (4mm,4mm)\rectanglepath;
						
					\end{tikzpicture}&

					\quad
					\begin{tikzpicture}[scale = 1]
						\def\rectanglepath{[red]-- +(8mm,0mm) -- +(8mm,8mm) -- +(0mm,8mm) -- cycle}

						%------
						\draw [step=4mm,dotted] (-4mm,0mm) grid (24mm,24mm);
						
						\draw [ultra thick][black](8mm,10mm)--(12mm,10mm);
						
						\draw[thick](2mm,0mm)--(2mm,6mm)--(6mm,6mm)--(6mm,10mm)--(10mm,10mm)--(14mm,10mm)--(14mm,14mm)--(18mm,14mm)--(18mm,18mm)--(24mm,18mm);
						\draw (8mm,8mm)\rectanglepath;
						
					\end{tikzpicture}&
					
					\quad
					\begin{tikzpicture}[scale = 1]
						\def\rectanglepath{[red]-- +(8mm,0mm) -- +(8mm,8mm) -- +(0mm,8mm) -- cycle}

						%------
						\draw [step=4mm,dotted] (-4mm,0mm) grid (24mm,24mm);
						
						\draw [ultra thick][black](12mm,14mm)--(16mm,14mm);
						
						\draw[thick](2mm,0mm)--(2mm,6mm)--(6mm,6mm)--(6mm,10mm)--(10mm,10mm)--(10mm,14mm)--(14mm,14mm)--(18mm,14mm)--(18mm,18mm)--(24mm,18mm);
						\draw (12mm,12mm)\rectanglepath;
						
					\end{tikzpicture}&

					\quad
					\begin{tikzpicture}[scale = 1]
						
						%------
						\draw [step=4mm,dotted] (-4mm,0mm) grid (24mm,24mm);
						
						\draw[thick](2mm,0mm)--(2mm,6mm)--(6mm,6mm)--(6mm,10mm)--(10mm,10mm)--(10mm,14mm)--(14mm,14mm)--(14mm,18mm)--(18mm,18mm)--(24mm,18mm);
						
					\end{tikzpicture}
					
				\end{tabular}
				\caption{An illustration of the undrooping procedure.}\label{undroop}
			\end{center}
		\end{figure}
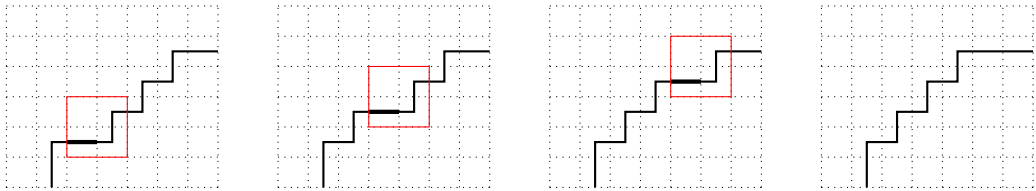
	\end{itemize}

	\begin{example}
    \label{e: algorithm}
Figure \ref{bumpless} illustrates the construction of the maximal marked bumpless pipedream associated with
$w=5241736.$
Starting from Figure~\ref{bumpless}(a), we apply the algorithm successively through Figures~\ref{bumpless}(b)--(e). Since no long line appears during the process, only the ``Drooping'' algorithm is required. In this figure, the red line marks the pipe that will execute the algorithm. 
			\begin{figure}[h]
				\begin{center}
					\begin{tabular}{ccc}
						%%%1
						
						\begin{tikzpicture}[scale = 1]
                            \def\rectanglepath{[fill=gray!20!white,draw=black!20!black]-- +(4mm,0mm) -- +(4mm,4mm) -- +(0mm,4mm) -- cycle}
							
							\def\squarepath{-- +(28mm,0mm) -- +(28mm,28mm) -- +(0mm,28mm) -- cycle}
							
							%------
							\draw (0mm,0mm)\squarepath;
							\draw [step=4mm,dotted] (0mm,0mm) grid (28mm,28mm);
							\draw [thin](0mm,24mm) \rectanglepath;
							\draw [thin](0mm,20mm)\rectanglepath;
							\draw [thin](0mm,16mm)\rectanglepath;
							\draw [thin](4mm,24mm)\rectanglepath;
							\draw [thin](8mm,16mm)\rectanglepath;
							\draw [thin](8mm,8mm)\rectanglepath;
							\draw [thin](8mm,24mm)\rectanglepath;
							\draw [thin](12mm,24mm)\rectanglepath;
							\draw [thin](20mm,8mm)\rectanglepath;
							
							\draw[red,thick](2mm,0mm)--(2mm,14mm)--(28mm,14mm);
							\draw[thick](6mm,0mm)--(6mm,22mm)--(28mm,22mm);
							\draw[thick](10mm,0mm)--(10mm,6mm)--(28mm,6mm);
							\draw[thick](14mm,0mm)--(14mm,18mm)--(28mm,18mm);
							\draw[thick](18mm,0mm)--(18mm,26mm)--(28mm,26mm);
							\draw[thick](22mm,0mm)--(22mm,2mm)--(28mm,2mm);
							\draw[thick](26mm,0mm)--(26mm,10mm)--(28mm,10mm);
							
							%*
							\node at (22mm,14mm) {{\large $*$} };
							\node at (22mm,18mm) {{\large $*$} };
							\node at (10mm,22mm) {{\large $*$} };
							\node at (22mm,22mm) {{\large $*$} };
							\node at (22mm,26mm) {{\large $*$} };
							
							%column number
							\node at (30mm,26mm) {\small{5}};
							\node at (30mm,22mm) {\small{2}};
							\node at (30mm,18mm) {\small{4}};
							\node at (30mm,14mm) {\small{1}};
							\node at (30mm,10mm) {\small{7}};
							\node at (30mm,6mm) {\small{3}};
							\node at (30mm,2mm) {\small{6}};
							
							%row number
							\node at (2mm,-2mm) {\small{1}};
							\node at (6mm,-2mm) {\small{2}};
							\node at (10mm,-2mm) {\small{3}};
							\node at (14mm,-2mm)  {\small{4}};
							\node at (18mm,-2mm) {\small{5}};
							\node at (22mm,-2mm) {\small{6}};
							\node at (26mm,-2mm) {\small{7}};

							\node at (14mm,-7mm) {(a)};
						\end{tikzpicture}
						&

						%%%2
						\begin{tikzpicture}[scale = 1]
							
							\def\rectanglepath{[fill=gray!20!white,draw=black!20!black]-- +(4mm,0mm) -- +(4mm,4mm) -- +(0mm,4mm) -- cycle}
							
							\def\squarepath{-- +(28mm,0mm) -- +(28mm,28mm) -- +(0mm,28mm) -- cycle}

							\draw (0mm,0mm)\squarepath;
							\draw [step=4mm,dotted] (0mm,0mm) grid (28mm,28mm);
							\draw (0mm,24mm) \rectanglepath;
							\draw (0mm,20mm)\rectanglepath;
							\draw (0mm,16mm)\rectanglepath;
							\draw (0mm,12mm)\rectanglepath;
							\draw (4mm,24mm)\rectanglepath;
							\draw (8mm,16mm)\rectanglepath;
							\draw (8mm,24mm)\rectanglepath;
							\draw (12mm,24mm)\rectanglepath;
							\draw (20mm,8mm)\rectanglepath;
							
							\draw[thick](2mm,0mm)--(2mm,10mm)--(10mm,10mm)--(10mm,14mm)--(28mm,14mm);
							\draw[thick](6mm,0mm)--(6mm,22mm)--(28mm,22mm);
							\draw[thick](10mm,0mm)--(10mm,6mm)--(28mm,6mm);
							\draw[thick,red](14mm,0mm)--(14mm,18mm)--(28mm,18mm);
							\draw[thick](18mm,0mm)--(18mm,26mm)--(28mm,26mm);
							\draw[thick](22mm,0mm)--(22mm,2mm)--(28mm,2mm);
							\draw[thick](26mm,0mm)--(26mm,10mm)--(28mm,10mm);

							\node at (22mm,18mm) {{\large $*$} };
							\node at (22mm,22mm) {{\large $*$} };
							\node at (22mm,26mm) {{\large $*$} };
							\node at (10mm,22mm) {{\large $*$} };
							
							%column number
							\node at (30mm,26mm) {\small{5}};
							\node at (30mm,22mm) {\small{2}};
							\node at (30mm,18mm) {\small{4}};
							\node at (30mm,14mm) {\small{1}};
							\node at (30mm,10mm) {\small{7}};
							\node at (30mm,6mm) {\small{3}};
							\node at (30mm,2mm) {\small{6}};
							
							%row number
							\node at (2mm,-2mm) {\small{1}};
							\node at (6mm,-2mm) {\small{2}};
							\node at (10mm,-2mm) {\small{3}};
							\node at (14mm,-2mm)  {\small{4}};
							\node at (18mm,-2mm) {\small{5}};
							\node at (22mm,-2mm) {\small{6}};
							\node at (26mm,-2mm) {\small{7}};
							
							\filldraw[black] (10mm,10mm) circle (1.5pt);
							
							\node at (14mm,-7mm) {(b)};
						\end{tikzpicture}
						&

						%%%3
						\begin{tikzpicture}[scale = 1]
							
							\def\rectanglepath{[fill=gray!20!white,draw=black!20!black]-- +(4mm,0mm) -- +(4mm,4mm) -- +(0mm,4mm) -- cycle}
							
							\def\squarepath{-- +(28mm,0mm) -- +(28mm,28mm) -- +(0mm,28mm) -- cycle}

							\draw (0mm,0mm)\squarepath;
							\draw [step=4mm,dotted] (0mm,0mm) grid (28mm,28mm);
							\draw (0mm,24mm) \rectanglepath;
							\draw (0mm,20mm)\rectanglepath;
							\draw (0mm,16mm)\rectanglepath;
							\draw (0mm,12mm)\rectanglepath;
							\draw (4mm,24mm)\rectanglepath;
							\draw (8mm,16mm)\rectanglepath;
							\draw (8mm,24mm)\rectanglepath;
							\draw (12mm,24mm)\rectanglepath;
							\draw (12mm,16mm)\rectanglepath;

							\draw[thick](2mm,0mm)--(2mm,10mm)--(10mm,10mm)--(10mm,14mm)--(28mm,14mm);
							\draw[thick,red](6mm,0mm)--(6mm,22mm)--(28mm,22mm);
							\draw[thick](10mm,0mm)--(10mm,6mm)--(28mm,6mm);
							\draw[thick](14mm,0mm)--(14mm,10mm)--(22mm,10mm)--(22mm,18mm)--(28mm,18mm);
							\draw[thick](18mm,0mm)--(18mm,26mm)--(28mm,26mm);
							\draw[thick](22mm,0mm)--(22mm,2mm)--(28mm,2mm);
							\draw[thick](26mm,0mm)--(26mm,10mm)--(28mm,10mm);

							\node at (22mm,22mm) {{\large $*$} };
							\node at (22mm,26mm) {{\large $*$} };
							\node at (10mm,22mm) {{\large $*$} };
							
							%column number
							\node at (30mm,26mm) {\small{5}};
							\node at (30mm,22mm) {\small{2}};
							\node at (30mm,18mm) {\small{4}};
							\node at (30mm,14mm) {\small{1}};
							\node at (30mm,10mm) {\small{7}};
							\node at (30mm,6mm) {\small{3}};
							\node at (30mm,2mm) {\small{6}};
							
							%row number
							\node at (2mm,-2mm) {\small{1}};
							\node at (6mm,-2mm) {\small{2}};
							\node at (10mm,-2mm) {\small{3}};
							\node at (14mm,-2mm)  {\small{4}};
							\node at (18mm,-2mm) {\small{5}};
							\node at (22mm,-2mm) {\small{6}};
							\node at (26mm,-2mm) {\small{7}};
							
							\filldraw[black] (10mm,10mm) circle (1.5pt);
    \filldraw[black] (22mm,10mm) circle (1.5pt);
							
							\node at (14mm,-7mm) {(c)};
						\end{tikzpicture}
					\end{tabular}

					\begin{tabular}{cc}
						%%%4
						\begin{tikzpicture}[scale = 1]
							
							\def\rectanglepath{[fill=gray!20!white,draw=black!20!black]-- +(4mm,0mm) -- +(4mm,4mm) -- +(0mm,4mm) -- cycle}
							
							\def\squarepath{-- +(28mm,0mm) -- +(28mm,28mm) -- +(0mm,28mm) -- cycle}

							\draw (0mm,0mm)\squarepath;
							\draw [step=4mm,dotted] (0mm,0mm) grid (28mm,28mm);
							\draw (0mm,24mm) \rectanglepath;
							\draw (0mm,20mm)\rectanglepath;	
							\draw (0mm,16mm)\rectanglepath;
							\draw (0mm,12mm)\rectanglepath;
							\draw (4mm,24mm)\rectanglepath;
							\draw (4mm,20mm)\rectanglepath;
							\draw (4mm,16mm)\rectanglepath;
							\draw (8mm,24mm)\rectanglepath;
							\draw (8mm,20mm)\rectanglepath;
							\draw (12mm,24mm)\rectanglepath;
							
							\draw[thick](2mm,0mm)--(2mm,10mm)--(10mm,10mm)--(10mm,14mm)--(28mm,14mm);
							\draw[thick](6mm,0mm)--(6mm,14mm)--(10mm,14mm)--(10mm,18mm)--(14mm,18mm)--(14mm,22mm)--(28mm,22mm);
							\draw[thick](10mm,0mm)--(10mm,6mm)--(28mm,6mm);
							\draw[thick](14mm,0mm)--(14mm,10mm)--(22mm,10mm)--(22mm,18mm)--(28mm,18mm);
							\draw[thick,red](18mm,0mm)--(18mm,26mm)--(28mm,26mm);
							\draw[thick](22mm,0mm)--(22mm,2mm)--(28mm,2mm);
							\draw[thick](26mm,0mm)--(26mm,10mm)--(28mm,10mm);

							\node at (22mm,26mm) {{\large $*$} };
							
							%column number
							\node at (30mm,26mm) {\small{5}};
							\node at (30mm,22mm) {\small{2}};
							\node at (30mm,18mm) {\small{4}};
							\node at (30mm,14mm) {\small{1}};
							\node at (30mm,10mm) {\small{7}};
							\node at (30mm,6mm) {\small{3}};
							\node at (30mm,2mm) {\small{6}};
							
							%row number
							\node at (2mm,-2mm) {\small{1}};
							\node at (6mm,-2mm) {\small{2}};
							\node at (10mm,-2mm) {\small{3}};
							\node at (14mm,-2mm)  {\small{4}};
							\node at (18mm,-2mm) {\small{5}};
							\node at (22mm,-2mm) {\small{6}};
							\node at (26mm,-2mm) {\small{7}};
							
							\filldraw[black] (10mm,10mm) circle (1.5pt);
    \filldraw[black] (14mm,18mm) circle (1.5pt);
    \filldraw[black] (22mm,10mm) circle (1.5pt);
							
							\node at (14mm,-7mm) {(d)};
						\end{tikzpicture}
						
						&

						%%%5
						\begin{tikzpicture}[scale = 1]
							
							\def\rectanglepath{[fill=gray!20!white,draw=black!20!black]-- +(4mm,0mm) -- +(4mm,4mm) -- +(0mm,4mm) -- cycle}
							
							\def\squarepath{-- +(28mm,0mm) -- +(28mm,28mm) -- +(0mm,28mm) -- cycle}

							\draw (0mm,0mm)\squarepath;
							\draw [step=4mm,dotted] (0mm,0mm) grid (28mm,28mm);
							\draw (0mm,24mm) \rectanglepath;
							\draw (0mm,20mm)\rectanglepath;	
							\draw (0mm,16mm)\rectanglepath;
							\draw (0mm,12mm)\rectanglepath;
							\draw (4mm,24mm)\rectanglepath;
							\draw (4mm,20mm)\rectanglepath;
							\draw (4mm,16mm)\rectanglepath;
							\draw (8mm,24mm)\rectanglepath;
							\draw (8mm,20mm)\rectanglepath;
							\draw (12mm,24mm)\rectanglepath;
							\draw (16mm,24mm)\rectanglepath;

							\draw[thick](2mm,0mm)--(2mm,10mm)--(10mm,10mm)--(10mm,14mm)--(28mm,14mm);
							\draw[thick](6mm,0mm)--(6mm,14mm)--(10mm,14mm)--(10mm,18mm)--(14mm,18mm)--(14mm,22mm)--(28mm,22mm);
							\draw[thick](10mm,0mm)--(10mm,6mm)--(28mm,6mm);
							\draw[thick](14mm,0mm)--(14mm,10mm)--(22mm,10mm)--(22mm,18mm)--(28mm,18mm);
							\draw[thick](18mm,0mm)--(18mm,18mm)--(22mm,18mm)--(22mm,26mm)--(28mm,26mm);
							\draw[thick](22mm,0mm)--(22mm,2mm)--(28mm,2mm);
							\draw[thick](26mm,0mm)--(26mm,10mm)--(28mm,10mm);
							
							%column number
							\node at (30mm,26mm) {\small{5}};
							\node at (30mm,22mm) {\small{2}};
							\node at (30mm,18mm) {\small{4}};
							\node at (30mm,14mm) {\small{1}};
							\node at (30mm,10mm) {\small{7}};
							\node at (30mm,6mm) {\small{3}};
							\node at (30mm,2mm) {\small{6}};
							
							%row number
							\node at (2mm,-2mm) {\small{1}};
							\node at (6mm,-2mm) {\small{2}};
							\node at (10mm,-2mm) {\small{3}};
							\node at (14mm,-2mm)  {\small{4}};
							\node at (18mm,-2mm) {\small{5}};
							\node at (22mm,-2mm) {\small{6}};
							\node at (26mm,-2mm) {\small{7}};
							
			\filldraw[black] (10mm,10mm) circle (1.5pt);
    \filldraw[black] (14mm,18mm) circle (1.5pt);
    \filldraw[black] (22mm,10mm) circle (1.5pt);

							\node at (14mm,-7mm) {(e)};
						\end{tikzpicture}
					\end{tabular}
					\vspace{-.3cm}
					\caption{The algorithm on the snow Rothe pipedream of $w=5241736$.}
					\label{bumpless}
				\end{center}
			\end{figure}
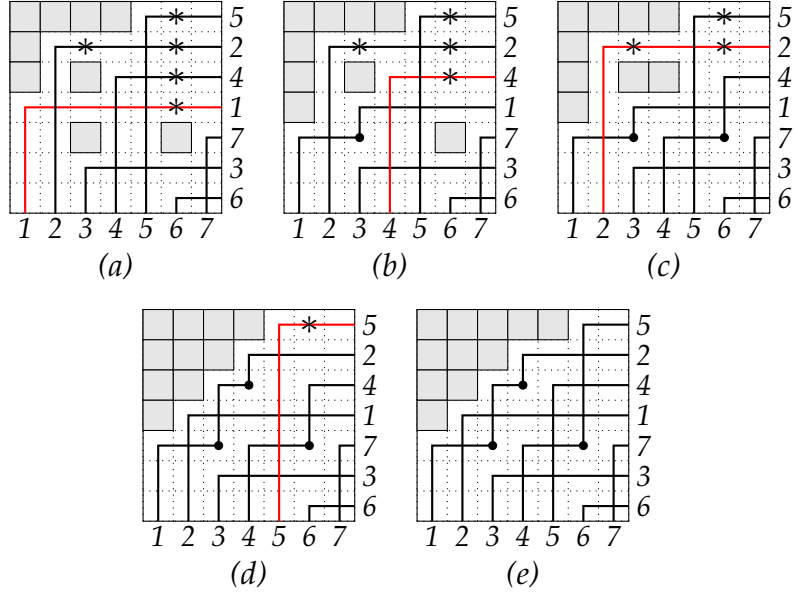
			
	\end{example}
    
    \vspace*{-10pt}
    
        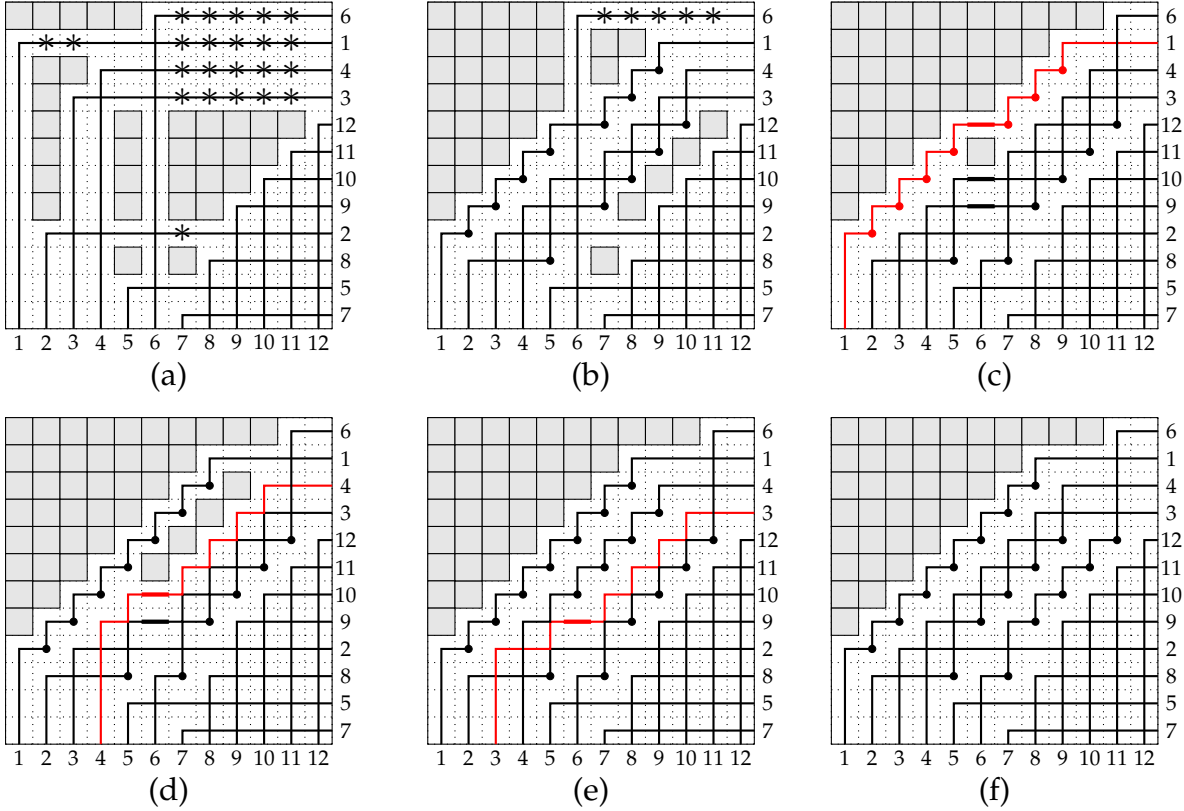
\begin{figure}[H]
			\begin{center}
				\begin{tabular}{ccc}
					\begin{tikzpicture}[scale = 0.9]
						
						\def\rectanglepath{[fill=gray!20!white,draw=black!20!black]-- +(4mm,0mm) -- +(4mm,4mm) -- +(0mm,4mm) -- cycle}
						
						\def\squarepath{-- +(48mm,0mm) -- +(48mm,48mm) -- +(0mm,48mm) -- cycle}
						
						\node at (26mm,14mm) {{\large $*$} };
						\node at (26mm,34mm) {{\large $*$} };
						\node at (26mm,38mm) {{\large $*$} };
						\node at (26mm,42mm) {{\large $*$} };
						\node at (26mm,46mm) {{\large $*$} };
						\node at (30mm,34mm) {{\large $*$} };
						\node at (30mm,38mm) {{\large $*$} };
						\node at (30mm,42mm) {{\large $*$} };
						\node at (30mm,46mm) {{\large $*$} };
						\node at (34mm,34mm) {{\large $*$} };
						\node at (34mm,38mm) {{\large $*$} };
						\node at (34mm,42mm) {{\large $*$} };
						\node at (34mm,46mm) {{\large $*$} };
						\node at (38mm,34mm) {{\large $*$} };
						\node at (38mm,38mm) {{\large $*$} };
						\node at (38mm,42mm) {{\large $*$} };
						\node at (38mm,46mm) {{\large $*$} };
						\node at (42mm,34mm) {{\large $*$} };
						\node at (42mm,38mm) {{\large $*$} };
						\node at (42mm,42mm) {{\large $*$} };
						\node at (42mm,46mm) {{\large $*$} };
						
						\node at (10mm,42mm) {{\large $*$} };
						\node at (6mm,42mm) {{\large $*$} };

						%%%1
						\draw (0mm,0mm)\squarepath;
						\draw [step=4mm,dotted] (0mm,0mm) grid (48mm,48mm);
						\draw (0mm,44mm)\rectanglepath;
						\draw (4mm,44mm)\rectanglepath;
						\draw (4mm,36mm)\rectanglepath;
						\draw (4mm,32mm)\rectanglepath;
						\draw (4mm,28mm)\rectanglepath;
						\draw (4mm,24mm)\rectanglepath;
						\draw (4mm,20mm)\rectanglepath;
						\draw (4mm,16mm)\rectanglepath;
						\draw (4mm,16mm)\rectanglepath;
						\draw (8mm,44mm)\rectanglepath;
						\draw (8mm,36mm)\rectanglepath;
						\draw (12mm,44mm)\rectanglepath;
						\draw (16mm,44mm)\rectanglepath;
						\draw (16mm,28mm)\rectanglepath;
						\draw (16mm,24mm)\rectanglepath;
						\draw (16mm,20mm)\rectanglepath;
						\draw (16mm,16mm)\rectanglepath;
						\draw (16mm,8mm)\rectanglepath;
						\draw (24mm,28mm)\rectanglepath;
						\draw (24mm,24mm)\rectanglepath;
						\draw (24mm,20mm)\rectanglepath;
						\draw (24mm,16mm)\rectanglepath;
						\draw (24mm,8mm)\rectanglepath;
						\draw (28mm,28mm)\rectanglepath;
						\draw (28mm,24mm)\rectanglepath;
						\draw (28mm,20mm)\rectanglepath;
						\draw (28mm,16mm)\rectanglepath;
						\draw (32mm,28mm)\rectanglepath;
						\draw (32mm,24mm)\rectanglepath;
						\draw (32mm,20mm)\rectanglepath;
						\draw (36mm,28mm)\rectanglepath;
						\draw (36mm,24mm)\rectanglepath;
						\draw (40mm,28mm)\rectanglepath;
						
						\draw[thick](2mm,0mm)--(2mm,42mm)--(48mm,42mm);
						\draw[thick](6mm,0mm)--(6mm,14mm)--(48mm,14mm);
						\draw[thick](10mm,0mm)--(10mm,34mm)--(48mm,34mm);
						\draw[thick](14mm,0mm)--(14mm,38mm)--(48mm,38mm);
						\draw[thick](18mm,0mm)--(18mm,6mm)--(48mm,6mm);
						\draw[thick](22mm,0mm)--(22mm,46mm)--(48mm,46mm);
						\draw[thick](26mm,0mm)--(26mm,2mm)--(48mm,2mm);
						\draw[thick](30mm,0mm)--(30mm,10mm)--(48mm,10mm);
						\draw[thick](34mm,0mm)--(34mm,18mm)--(48mm,18mm);
						\draw[thick](38mm,0mm)--(38mm,22mm)--(48mm,22mm);
						\draw[thick](42mm,0mm)--(42mm,26mm)--(48mm,26mm);
						\draw[thick](46mm,0mm)--(46mm,30mm)--(48mm,30mm);

						%column number
						\node at (50mm,46mm) {\tiny{6}};
						\node at (50mm,42mm) {\tiny{1}};
						\node at (50mm,38mm) {\tiny{4}};
						\node at (50mm,34mm) {\tiny{3}};
						\node at (50mm,30mm) {\tiny{12}};
						\node at (50mm,26mm) {\tiny{11}};
						\node at (50mm,22mm) {\tiny{10}};
						\node at (50mm,18mm) {\tiny{9}};
						\node at (50mm,14mm) {\tiny{2}};
						\node at (50mm,10mm) {\tiny{8}};
						\node at (50mm,6mm) {\tiny{5}};
						\node at (50mm,2mm) {\tiny{7}};
						%
						%row number
						\node at (2mm,-2mm) {\tiny{1}};
						\node at (6mm,-2mm) {\tiny{2}};
						\node at (10mm,-2mm) {\tiny{3}};
						\node at (14mm,-2mm)  {\tiny{4}};
						\node at (18mm,-2mm) {\tiny{5}};
						\node at (22mm,-2mm) {\tiny{6}};
						\node at (26mm,-2mm) {\tiny{7}};
						\node at (30mm,-2mm) {\tiny{8}};
						\node at (34mm,-2mm) {\tiny{9}};
						\node at (38mm,-2mm) {\tiny{10}};
						\node at (42mm,-2mm) {\tiny{11}};
						\node at (46mm,-2mm) {\tiny{12}};
						
						\node at (24mm,-7mm) {(a)};
						
					\end{tikzpicture}&

					\quad
					\begin{tikzpicture}[scale = 0.9]
						
						\def\rectanglepath{[fill=gray!20!white,draw=black!20!black]-- +(4mm,0mm) -- +(4mm,4mm) -- +(0mm,4mm) -- cycle}
						
						\def\squarepath{-- +(48mm,0mm) -- +(48mm,48mm) -- +(0mm,48mm) -- cycle}
						
						\node at (26mm,46mm) {{\large $*$} };
						\node at (30mm,46mm) {{\large $*$} };
						\node at (34mm,46mm) {{\large $*$} };
						\node at (38mm,46mm) {{\large $*$} };
						\node at (42mm,46mm) {{\large $*$} };
						
						%%%2
						\draw (0mm,0mm)\squarepath;
						\draw [step=4mm,dotted] (0mm,0mm) grid (48mm,48mm);
						\draw (0mm,44mm)\rectanglepath;
						\draw (0mm,40mm)\rectanglepath;
						\draw (0mm,36mm)\rectanglepath;
						\draw (0mm,32mm)\rectanglepath;
						\draw (0mm,28mm)\rectanglepath;
						\draw (0mm,24mm)\rectanglepath;
						\draw (0mm,20mm)\rectanglepath;
						\draw (0mm,16mm)\rectanglepath;
						
						\draw (4mm,44mm)\rectanglepath;
						\draw (4mm,40mm)\rectanglepath;
						\draw (4mm,36mm)\rectanglepath;
						\draw (4mm,32mm)\rectanglepath;
						\draw (4mm,28mm)\rectanglepath;
						\draw (4mm,20mm)\rectanglepath;
						\draw (4mm,24mm)\rectanglepath;
						%\draw (4mm,12mm)\rectanglepath;
						
						\draw (8mm,44mm)\rectanglepath;
						\draw (8mm,40mm)\rectanglepath;
						\draw (8mm,36mm)\rectanglepath;
						\draw (8mm,32mm)\rectanglepath;
						\draw (8mm,24mm)\rectanglepath;
						\draw (8mm,28mm)\rectanglepath;
						%\draw (8mm,16mm)\rectanglepath;
						
						\draw (12mm,44mm)\rectanglepath;
						\draw (12mm,40mm)\rectanglepath;
						\draw (12mm,36mm)\rectanglepath;
						\draw (12mm,28mm)\rectanglepath;
						\draw (12mm,32mm)\rectanglepath;
						%\draw (12mm,20mm)\rectanglepath;
						
						\draw (16mm,44mm)\rectanglepath;
						\draw (16mm,40mm)\rectanglepath;
						\draw (16mm,32mm)\rectanglepath;
						\draw (16mm,36mm)\rectanglepath;
						%\draw (16mm,24mm)\rectanglepath;
						%\draw (16mm,8mm)\rectanglepath;

						\draw (24mm,36mm)\rectanglepath;
						\draw (24mm,40mm)\rectanglepath;
						%\draw (24mm,28mm)\rectanglepath;
						\draw (24mm,8mm)\rectanglepath;
						%\draw (24mm,16mm)\rectanglepath;
						
						\draw (28mm,40mm)\rectanglepath;
						%\draw (28mm,32mm)\rectanglepath;
						\draw (28mm,16mm)\rectanglepath;
						%\draw (28mm,20mm)\rectanglepath;
						
						%\draw (32mm,36mm)\rectanglepath;
						\draw (32mm,20mm)\rectanglepath;
						%\draw (32mm,24mm)\rectanglepath;
						
						\draw (36mm,24mm)\rectanglepath;
						%\draw (36mm,28mm)\rectanglepath;
						
						\draw (40mm,28mm)\rectanglepath;
						
						\draw[thick](2mm,0mm)--(2mm,14mm)--(6mm,14mm)--(6mm,18mm)--(10mm,18mm)--(10mm,22mm)--(14mm,22mm)--(14mm,26mm)--(18mm,26mm)--(18mm,30mm)--(26mm,30mm)--(26mm,34mm)--(30mm,34mm)--(30mm,38mm)--(34mm,38mm)--(34mm,42mm)--(48mm,42mm);
						\draw[thick](6mm,0mm)--(6mm,10mm)--(18mm,10mm)--(18mm,14mm)--(48mm,14mm);
						\draw[thick](10mm,0mm)--(10mm,14mm)--(18mm,14mm)--(18mm,18mm)--(26mm,18mm)--(26mm,22mm)--(30mm,22mm)--(30mm,26mm)--(34mm,26mm)--(34mm,30mm)--(38mm,30mm)--(38mm,34mm)--(48mm,34mm);
						\draw[thick](14mm,0mm)--(14mm,18mm)--(18mm,18mm)--(18mm,22mm)--(26mm,22mm)--(26mm,26mm)--(30mm,26mm)--(30mm,30mm)--(34mm,30mm)--(34mm,34mm)--(38mm,34mm)--(38mm,38mm)--(48mm,38mm);
						\draw[thick](18mm,0mm)--(18mm,6mm)--(48mm,6mm);
						\draw[thick](22mm,0mm)--(22mm,46mm)--(48mm,46mm);
						\draw[thick](26mm,0mm)--(26mm,2mm)--(48mm,2mm);
						\draw[thick](30mm,0mm)--(30mm,10mm)--(48mm,10mm);
						\draw[thick](34mm,0mm)--(34mm,18mm)--(48mm,18mm);
						\draw[thick](38mm,0mm)--(38mm,22mm)--(48mm,22mm);
						\draw[thick](42mm,0mm)--(42mm,26mm)--(48mm,26mm);
						\draw[thick](46mm,0mm)--(46mm,30mm)--(48mm,30mm);

						%column number
						\node at (50mm,46mm) {\tiny{6}};
						\node at (50mm,42mm) {\tiny{1}};
						\node at (50mm,38mm) {\tiny{4}};
						\node at (50mm,34mm) {\tiny{3}};
						\node at (50mm,30mm) {\tiny{12}};
						\node at (50mm,26mm) {\tiny{11}};
						\node at (50mm,22mm) {\tiny{10}};
						\node at (50mm,18mm) {\tiny{9}};
						\node at (50mm,14mm) {\tiny{2}};
						\node at (50mm,10mm) {\tiny{8}};
						\node at (50mm,6mm) {\tiny{5}};
						\node at (50mm,2mm) {\tiny{7}};
						%
						%row number
						\node at (2mm,-2mm) {\tiny{1}};
						\node at (6mm,-2mm) {\tiny{2}};
						\node at (10mm,-2mm) {\tiny{3}};
						\node at (14mm,-2mm)  {\tiny{4}};
						\node at (18mm,-2mm) {\tiny{5}};
						\node at (22mm,-2mm) {\tiny{6}};
						\node at (26mm,-2mm) {\tiny{7}};
						\node at (30mm,-2mm) {\tiny{8}};
						\node at (34mm,-2mm) {\tiny{9}};
						\node at (38mm,-2mm) {\tiny{10}};
						\node at (42mm,-2mm) {\tiny{11}};
						\node at (46mm,-2mm) {\tiny{12}};
						
						\node at (24mm,-7mm) {(b)};
						\filldraw[black] (6mm,14mm) circle (1.5pt);
						\filldraw[black] (10mm,18mm) circle (1.5pt);
						\filldraw[black] (14mm,22mm) circle (1.5pt);
						\filldraw[black] (18mm,26mm) circle (1.5pt);
						\filldraw[black] (26mm,30mm) circle (1.5pt);
						\filldraw[black] (30mm,34mm) circle (1.5pt);
						\filldraw[black] (34mm,38mm) circle (1.5pt);
						\filldraw[black] (18mm,10mm) circle (1.5pt);
						\filldraw[black] (26mm,18mm) circle (1.5pt);
						\filldraw[black] (30mm,22mm) circle (1.5pt);
						\filldraw[black] (34mm,26mm) circle (1.5pt);
						\filldraw[black] (38mm,30mm) circle (1.5pt);
					\end{tikzpicture}
                    \quad
					%%%%%%%%%%%%%%%%%%%%%%3
					\begin{tikzpicture}[scale = 0.9]
						
						\def\rectanglepath{[fill=gray!20!white,draw=black!20!black]-- +(4mm,0mm) -- +(4mm,4mm) -- +(0mm,4mm) -- cycle}
						
						\def\squarepath{-- +(48mm,0mm) -- +(48mm,48mm) -- +(0mm,48mm) -- cycle}

						%%%2
						\draw (0mm,0mm)\squarepath;
						\draw [step=4mm,dotted] (0mm,0mm) grid (48mm,48mm);
						\draw (0mm,44mm)\rectanglepath;
						\draw (0mm,40mm)\rectanglepath;
						\draw (0mm,36mm)\rectanglepath;
						\draw (0mm,32mm)\rectanglepath;
						\draw (0mm,28mm)\rectanglepath;
						\draw (0mm,24mm)\rectanglepath;
						\draw (0mm,20mm)\rectanglepath;
						\draw (0mm,16mm)\rectanglepath;
						
						\draw (4mm,44mm)\rectanglepath;
						\draw (4mm,40mm)\rectanglepath;
						\draw (4mm,36mm)\rectanglepath;
						\draw (4mm,32mm)\rectanglepath;
						\draw (4mm,28mm)\rectanglepath;
						\draw (4mm,20mm)\rectanglepath;
						\draw (4mm,24mm)\rectanglepath;
						%\draw (4mm,12mm)\rectanglepath;
						
						\draw (8mm,44mm)\rectanglepath;
						\draw (8mm,40mm)\rectanglepath;
						\draw (8mm,36mm)\rectanglepath;
						\draw (8mm,32mm)\rectanglepath;
						\draw (8mm,24mm)\rectanglepath;
						\draw (8mm,28mm)\rectanglepath;
						%\draw (8mm,16mm)\rectanglepath;
						
						\draw (12mm,44mm)\rectanglepath;
						\draw (12mm,40mm)\rectanglepath;
						\draw (12mm,36mm)\rectanglepath;
						\draw (12mm,28mm)\rectanglepath;
						\draw (12mm,32mm)\rectanglepath;
						%\draw (12mm,20mm)\rectanglepath;
						
						\draw (16mm,44mm)\rectanglepath;
						\draw (16mm,40mm)\rectanglepath;
						\draw (16mm,32mm)\rectanglepath;
						\draw (16mm,36mm)\rectanglepath;
						%\draw (16mm,24mm)\rectanglepath;
						%\draw (16mm,8mm)\rectanglepath;
						
						\draw (20mm,44mm)\rectanglepath;
						\draw (20mm,40mm)\rectanglepath;
						\draw (20mm,32mm)\rectanglepath;
						\draw (20mm,36mm)\rectanglepath;
						\draw (20mm,24mm)\rectanglepath;
						
						\draw (24mm,44mm)\rectanglepath;
						\draw (24mm,36mm)\rectanglepath;
						\draw (24mm,40mm)\rectanglepath;
						%\draw (24mm,28mm)\rectanglepath;
						%\draw (24mm,8mm)\rectanglepath;
						
						\draw (28mm,44mm)\rectanglepath;
						\draw (28mm,40mm)\rectanglepath;
						%\draw (28mm,32mm)\rectanglepath;
						%\draw (28mm,16mm)\rectanglepath;
						
						\draw (32mm,44mm)\rectanglepath;
						%\draw (32mm,36mm)\rectanglepath;
						%\draw (32mm,20mm)\rectanglepath;
						
						\draw (36mm,44mm)\rectanglepath;
						%\draw (36mm,24mm)\rectanglepath;
						
						%\draw (40mm,28mm)\rectanglepath;
						
						\draw[red,thick](2mm,0mm)--(2mm,14mm)--(6mm,14mm)--(6mm,18mm)--(10mm,18mm)--(10mm,22mm)--(14mm,22mm)--(14mm,26mm)--(18mm,26mm)--(18mm,30mm)--(26mm,30mm)--(26mm,34mm)--(30mm,34mm)--(30mm,38mm)--(34mm,38mm)--(34mm,42mm)--(48mm,42mm);
						\draw[thick](6mm,0mm)--(6mm,10mm)--(18mm,10mm)--(18mm,14mm)--(48mm,14mm);
						\draw[thick](10mm,0mm)--(10mm,14mm)--(18mm,14mm)--(18mm,18mm)--(26mm,18mm)--(26mm,22mm)--(30mm,22mm)--(30mm,26mm)--(34mm,26mm)--(34mm,30mm)--(38mm,30mm)--(38mm,34mm)--(48mm,34mm);
						\draw[thick](14mm,0mm)--(14mm,18mm)--(18mm,18mm)--(18mm,22mm)--(26mm,22mm)--(26mm,26mm)--(30mm,26mm)--(30mm,30mm)--(34mm,30mm)--(34mm,34mm)--(38mm,34mm)--(38mm,38mm)--(48mm,38mm);
						\draw[thick](18mm,0mm)--(18mm,6mm)--(48mm,6mm);
						\draw[thick](22mm,0mm)--(22mm,10mm)--(26mm,10mm)--(26mm,18mm)--(30mm,18mm)--(30mm,22mm)--(34mm,22mm)--(34mm,26mm)--(38mm,26mm)--(38mm,30mm)--(42mm,30mm)--(42mm,46mm)--(48mm,46mm);
						\draw[thick](26mm,0mm)--(26mm,2mm)--(48mm,2mm);
						\draw[thick](30mm,0mm)--(30mm,10mm)--(48mm,10mm);
						\draw[thick](34mm,0mm)--(34mm,18mm)--(48mm,18mm);
						\draw[thick](38mm,0mm)--(38mm,22mm)--(48mm,22mm);
						\draw[thick](42mm,0mm)--(42mm,26mm)--(48mm,26mm);
						\draw[thick](46mm,0mm)--(46mm,30mm)--(48mm,30mm);
						
						\draw [ultra thick][red](20mm,30mm)--(24mm,30mm);
						\draw [ultra thick][black](20mm,22mm)--(24mm,22mm);
						\draw [ultra thick][black](20mm,18mm)--(24mm,18mm);

						%column number
						\node at (50mm,46mm) {\tiny{6}};
						\node at (50mm,42mm) {\tiny{1}};
						\node at (50mm,38mm) {\tiny{4}};
						\node at (50mm,34mm) {\tiny{3}};
						\node at (50mm,30mm) {\tiny{12}};
						\node at (50mm,26mm) {\tiny{11}};
						\node at (50mm,22mm) {\tiny{10}};
						\node at (50mm,18mm) {\tiny{9}};
						\node at (50mm,14mm) {\tiny{2}};
						\node at (50mm,10mm) {\tiny{8}};
						\node at (50mm,6mm) {\tiny{5}};
						\node at (50mm,2mm) {\tiny{7}};
						%
						%row number
						\node at (2mm,-2mm) {\tiny{1}};
						\node at (6mm,-2mm) {\tiny{2}};
						\node at (10mm,-2mm) {\tiny{3}};
						\node at (14mm,-2mm)  {\tiny{4}};
						\node at (18mm,-2mm) {\tiny{5}};
						\node at (22mm,-2mm) {\tiny{6}};
						\node at (26mm,-2mm) {\tiny{7}};
						\node at (30mm,-2mm) {\tiny{8}};
						\node at (34mm,-2mm) {\tiny{9}};
						\node at (38mm,-2mm) {\tiny{10}};
						\node at (42mm,-2mm) {\tiny{11}};
						\node at (46mm,-2mm) {\tiny{12}};
						\filldraw[red] (6mm,14mm) circle (1.5pt);
						\filldraw[red] (10mm,18mm) circle (1.5pt);
						\filldraw[red] (14mm,22mm) circle (1.5pt);
						\filldraw[red] (18mm,26mm) circle (1.5pt);
						\filldraw[red] (26mm,30mm) circle (1.5pt);
						\filldraw[red] (30mm,34mm) circle (1.5pt);
						\filldraw[red] (34mm,38mm) circle (1.5pt);
						\filldraw[black] (18mm,10mm) circle (1.5pt);
						
						\filldraw[black] (26mm,10mm) circle (1.5pt);
						\filldraw[black] (30mm,18mm) circle (1.5pt);
						\filldraw[black] (34mm,22mm) circle (1.5pt);
						\filldraw[black] (38mm,26mm) circle (1.5pt);
						\filldraw[black] (42mm,30mm) circle (1.5pt);
						\node at (24mm,-7mm) {(c)};
						
					\end{tikzpicture} \\
					
					%%%%%%%%%%%%%%%%%%4			
					
					\begin{tikzpicture}[scale = 0.9]
						
						\def\rectanglepath{[fill=gray!20!white,draw=black!20!black]-- +(4mm,0mm) -- +(4mm,4mm) -- +(0mm,4mm) -- cycle}
						
						\def\squarepath{-- +(48mm,0mm) -- +(48mm,48mm) -- +(0mm,48mm) -- cycle}

						%%%3
						\draw (0mm,0mm)\squarepath;
						\draw [step=4mm,dotted] (0mm,0mm) grid (48mm,48mm);
						\draw (0mm,44mm)\rectanglepath;
						\draw (0mm,40mm)\rectanglepath;
						\draw (0mm,36mm)\rectanglepath;
						\draw (0mm,32mm)\rectanglepath;
						\draw (0mm,28mm)\rectanglepath;
						\draw (0mm,24mm)\rectanglepath;
						\draw (0mm,20mm)\rectanglepath;
						\draw (0mm,16mm)\rectanglepath;
						
						\draw (4mm,44mm)\rectanglepath;
						\draw (4mm,40mm)\rectanglepath;
						\draw (4mm,36mm)\rectanglepath;
						\draw (4mm,32mm)\rectanglepath;
						\draw (4mm,28mm)\rectanglepath;
						\draw (4mm,20mm)\rectanglepath;
						\draw (4mm,24mm)\rectanglepath;
						%\draw (4mm,12mm)\rectanglepath;   % 黑点 (6,14)
						
						\draw (8mm,44mm)\rectanglepath;
						\draw (8mm,40mm)\rectanglepath;
						\draw (8mm,36mm)\rectanglepath;
						\draw (8mm,32mm)\rectanglepath;
						\draw (8mm,24mm)\rectanglepath;
						\draw (8mm,28mm)\rectanglepath;
						%\draw (8mm,16mm)\rectanglepath;   % 黑点 (10,18)
						
						\draw (12mm,44mm)\rectanglepath;
						\draw (12mm,40mm)\rectanglepath;
						\draw (12mm,36mm)\rectanglepath;
						\draw (12mm,28mm)\rectanglepath;
						\draw (12mm,32mm)\rectanglepath;
						%\draw (12mm,20mm)\rectanglepath;   % 黑点 (14,22)
						
						\draw (16mm,44mm)\rectanglepath;
						\draw (16mm,40mm)\rectanglepath;
						\draw (16mm,32mm)\rectanglepath;
						\draw (16mm,36mm)\rectanglepath;
						%\draw (16mm,24mm)\rectanglepath;   % 黑点 (18,26)
						%\draw (16mm,8mm)\rectanglepath;    % 黑点 (18,10)
						
						\draw (20mm,44mm)\rectanglepath;
						\draw (20mm,40mm)\rectanglepath;
						\draw (20mm,36mm)\rectanglepath;
						%\draw (20mm,28mm)\rectanglepath;
						\draw (20mm,24mm)\rectanglepath;
						
						\draw (24mm,44mm)\rectanglepath;
						\draw (24mm,40mm)\rectanglepath;
						%\draw (24mm,32mm)\rectanglepath;
						\draw (24mm,28mm)\rectanglepath;
						%\draw (24mm,8mm)\rectanglepath;    % 黑点 (26,10)
						
						\draw (28mm,44mm)\rectanglepath;
						%\draw (28mm,36mm)\rectanglepath;
						\draw (28mm,32mm)\rectanglepath;
						%\draw (28mm,16mm)\rectanglepath;   % 黑点 (30,18)
						
						\draw (32mm,44mm)\rectanglepath;
						\draw (32mm,36mm)\rectanglepath;
						%\draw (32mm,20mm)\rectanglepath;   % 黑点 (34,22)
						
						\draw (36mm,44mm)\rectanglepath;
						%\draw (36mm,24mm)\rectanglepath;   % 黑点 (38,26)
						
						%\draw (40mm,28mm)\rectanglepath;   % 黑点 (42,30)
						
						\draw[thick](2mm,0mm)--(2mm,14mm)--(6mm,14mm)--(6mm,18mm)--(10mm,18mm)--(10mm,22mm)--(14mm,22mm)--(14mm,26mm)--(18mm,26mm)--(18mm,30mm)--(22mm,30mm)--(22mm,34mm)--(26mm,34mm)--(26mm,38mm)--(30mm,38mm)--(30mm,42mm)--(48mm,42mm);
						\draw[thick](6mm,0mm)--(6mm,10mm)--(18mm,10mm)--(18mm,14mm)--(48mm,14mm);
						\draw[thick](10mm,0mm)--(10mm,14mm)--(18mm,14mm)--(18mm,18mm)--(26mm,18mm)--(26mm,22mm)--(30mm,22mm)--(30mm,26mm)--(34mm,26mm)--(34mm,30mm)--(38mm,30mm)--(38mm,34mm)--(48mm,34mm);
						\draw[red,thick](14mm,0mm)--(14mm,18mm)--(18mm,18mm)--(18mm,22mm)--(26mm,22mm)--(26mm,26mm)--(30mm,26mm)--(30mm,30mm)--(34mm,30mm)--(34mm,34mm)--(38mm,34mm)--(38mm,38mm)--(48mm,38mm);
						\draw[thick](18mm,0mm)--(18mm,6mm)--(48mm,6mm);
						\draw[thick](22mm,0mm)--(22mm,10mm)--(26mm,10mm)--(26mm,18mm)--(30mm,18mm)--(30mm,22mm)--(34mm,22mm)--(34mm,26mm)--(38mm,26mm)--(38mm,30mm)--(42mm,30mm)--(42mm,46mm)--(48mm,46mm);
						\draw[thick](26mm,0mm)--(26mm,2mm)--(48mm,2mm);
						\draw[thick](30mm,0mm)--(30mm,10mm)--(48mm,10mm);
						\draw[thick](34mm,0mm)--(34mm,18mm)--(48mm,18mm);
						\draw[thick](38mm,0mm)--(38mm,22mm)--(48mm,22mm);
						\draw[thick](42mm,0mm)--(42mm,26mm)--(48mm,26mm);
						\draw[thick](46mm,0mm)--(46mm,30mm)--(48mm,30mm);

						\draw [ultra thick][red](20mm,22mm)--(24mm,22mm);
						\draw [ultra thick][black](20mm,18mm)--(24mm,18mm);
						
						%column number
						\node at (50mm,46mm) {\tiny{6}};
						\node at (50mm,42mm) {\tiny{1}};
						\node at (50mm,38mm) {\tiny{4}};
						\node at (50mm,34mm) {\tiny{3}};
						\node at (50mm,30mm) {\tiny{12}};
						\node at (50mm,26mm) {\tiny{11}};
						\node at (50mm,22mm) {\tiny{10}};
						\node at (50mm,18mm) {\tiny{9}};
						\node at (50mm,14mm) {\tiny{2}};
						\node at (50mm,10mm) {\tiny{8}};
						\node at (50mm,6mm) {\tiny{5}};
						\node at (50mm,2mm) {\tiny{7}};
						%
						%row number
						\node at (2mm,-2mm) {\tiny{1}};
						\node at (6mm,-2mm) {\tiny{2}};
						\node at (10mm,-2mm) {\tiny{3}};
						\node at (14mm,-2mm)  {\tiny{4}};
						\node at (18mm,-2mm) {\tiny{5}};
						\node at (22mm,-2mm) {\tiny{6}};
						\node at (26mm,-2mm) {\tiny{7}};
						\node at (30mm,-2mm) {\tiny{8}};
						\node at (34mm,-2mm) {\tiny{9}};
						\node at (38mm,-2mm) {\tiny{10}};
						\node at (42mm,-2mm) {\tiny{11}};
						\node at (46mm,-2mm) {\tiny{12}};
						
						\node at (24mm,-7mm) {(d)};
						\filldraw[black] (6mm,14mm) circle (1.5pt);
						\filldraw[black] (10mm,18mm) circle (1.5pt);
						\filldraw[black] (14mm,22mm) circle (1.5pt);
						\filldraw[black] (18mm,26mm) circle (1.5pt);
						\filldraw[black] (22mm,30mm) circle (1.5pt);
						\filldraw[black] (26mm,34mm) circle (1.5pt);
						\filldraw[black] (30mm,38mm) circle (1.5pt);
						\filldraw[black] (18mm,10mm) circle (1.5pt);
						
						\filldraw[black] (26mm,10mm) circle (1.5pt);
                        \filldraw[black] (30mm,18mm) circle (1.5pt);
						\filldraw[black] (34mm,22mm) circle (1.5pt);
						\filldraw[black] (38mm,26mm) circle (1.5pt);
						\filldraw[black] (42mm,30mm) circle (1.5pt);
						
					\end{tikzpicture}&

					\quad
					%%%5
					\begin{tikzpicture}[scale = 0.9]
						
						\def\rectanglepath{[fill=gray!20!white,draw=black!20!black]-- +(4mm,0mm) -- +(4mm,4mm) -- +(0mm,4mm) -- cycle}
						
						\def\squarepath{-- +(48mm,0mm) -- +(48mm,48mm) -- +(0mm,48mm) -- cycle}
						
						\draw (0mm,0mm)\squarepath;
						\draw [step=4mm,dotted] (0mm,0mm) grid (48mm,48mm);
						\draw (0mm,44mm)\rectanglepath;
						\draw (0mm,40mm)\rectanglepath;
						\draw (0mm,36mm)\rectanglepath;
						\draw (0mm,32mm)\rectanglepath;
						\draw (0mm,28mm)\rectanglepath;
						\draw (0mm,24mm)\rectanglepath;
						\draw (0mm,20mm)\rectanglepath;
						\draw (0mm,16mm)\rectanglepath;
						
						\draw (4mm,44mm)\rectanglepath;
						\draw (4mm,40mm)\rectanglepath;
						\draw (4mm,36mm)\rectanglepath;
						\draw (4mm,32mm)\rectanglepath;
						\draw (4mm,28mm)\rectanglepath;
						\draw (4mm,20mm)\rectanglepath;
						\draw (4mm,24mm)\rectanglepath;
						%\draw (4mm,12mm)\rectanglepath;
						
						\draw (8mm,44mm)\rectanglepath;
						\draw (8mm,40mm)\rectanglepath;
						\draw (8mm,36mm)\rectanglepath;
						\draw (8mm,32mm)\rectanglepath;
						\draw (8mm,24mm)\rectanglepath;
						\draw (8mm,28mm)\rectanglepath;
						%\draw (8mm,16mm)\rectanglepath;
						
						\draw (12mm,44mm)\rectanglepath;
						\draw (12mm,40mm)\rectanglepath;
						\draw (12mm,36mm)\rectanglepath;
						\draw (12mm,28mm)\rectanglepath;
						\draw (12mm,32mm)\rectanglepath;
						%\draw (12mm,20mm)\rectanglepath;
						
						\draw (16mm,44mm)\rectanglepath;
						\draw (16mm,40mm)\rectanglepath;
						\draw (16mm,32mm)\rectanglepath;
						\draw (16mm,36mm)\rectanglepath;
						%\draw (16mm,24mm)\rectanglepath;
						%\draw (16mm,8mm)\rectanglepath;
						
						\draw (20mm,44mm)\rectanglepath;
						\draw (20mm,40mm)\rectanglepath;
						\draw (20mm,36mm)\rectanglepath;
						%\draw (20mm,28mm)\rectanglepath;
						%\draw (20mm,20mm)\rectanglepath;
						
						\draw (24mm,44mm)\rectanglepath;
						\draw (24mm,40mm)\rectanglepath;
						%\draw (24mm,32mm)\rectanglepath;
						%\draw (24mm,24mm)\rectanglepath;
						%\draw (24mm,8mm)\rectanglepath;
						
						\draw (28mm,44mm)\rectanglepath;
						%\draw (28mm,36mm)\rectanglepath;
						%\draw (28mm,28mm)\rectanglepath;
						%\draw (28mm,16mm)\rectanglepath;
						
						\draw (32mm,44mm)\rectanglepath;
						%\draw (32mm,32mm)\rectanglepath;
						%\draw (32mm,20mm)\rectanglepath;
						
						\draw (36mm,44mm)\rectanglepath;
						%\draw (36mm,24mm)\rectanglepath;
						
						%\draw (40mm,28mm)\rectanglepath;
						
						\draw[thick](2mm,0mm)--(2mm,14mm)--(6mm,14mm)--(6mm,18mm)--(10mm,18mm)--(10mm,22mm)--(14mm,22mm)--(14mm,26mm)--(18mm,26mm)--(18mm,30mm)--(22mm,30mm)--(22mm,34mm)--(26mm,34mm)--(26mm,38mm)--(30mm,38mm)--(30mm,42mm)--(48mm,42mm);
						\draw[thick](6mm,0mm)--(6mm,10mm)--(18mm,10mm)--(18mm,14mm)--(48mm,14mm);
						\draw[red,thick](10mm,0mm)--(10mm,14mm)--(18mm,14mm)--(18mm,18mm)--(26mm,18mm)--(26mm,22mm)--(30mm,22mm)--(30mm,26mm)--(34mm,26mm)--(34mm,30mm)--(38mm,30mm)--(38mm,34mm)--(48mm,34mm);
						\draw[thick](14mm,0mm)--(14mm,18mm)--(18mm,18mm)--(18mm,22mm)--(22mm,22mm)--(22mm,26mm)--(26mm,26mm)--(26mm,30mm)--(30mm,30mm)--(30mm,34mm)--(34mm,34mm)--(34mm,38mm)--(48mm,38mm);
						\draw[thick](18mm,0mm)--(18mm,6mm)--(48mm,6mm);
						\draw[thick](22mm,0mm)--(22mm,10mm)--(26mm,10mm)--(26mm,18mm)--(30mm,18mm)--(30mm,22mm)--(34mm,22mm)--(34mm,26mm)--(38mm,26mm)--(38mm,30mm)--(42mm,30mm)--(42mm,46mm)--(48mm,46mm);
						\draw[thick](26mm,0mm)--(26mm,2mm)--(48mm,2mm);
						\draw[thick](30mm,0mm)--(30mm,10mm)--(48mm,10mm);
						\draw[thick](34mm,0mm)--(34mm,18mm)--(48mm,18mm);
						\draw[thick](38mm,0mm)--(38mm,22mm)--(48mm,22mm);
						\draw[thick](42mm,0mm)--(42mm,26mm)--(48mm,26mm);
						\draw[thick](46mm,0mm)--(46mm,30mm)--(48mm,30mm);

						\draw [ultra thick][red](20mm,18mm)--(24mm,18mm);

						%column number
						\node at (50mm,46mm) {\tiny{6}};
						\node at (50mm,42mm) {\tiny{1}};
						\node at (50mm,38mm) {\tiny{4}};
						\node at (50mm,34mm) {\tiny{3}};
						\node at (50mm,30mm) {\tiny{12}};
						\node at (50mm,26mm) {\tiny{11}};
						\node at (50mm,22mm) {\tiny{10}};
						\node at (50mm,18mm) {\tiny{9}};
						\node at (50mm,14mm) {\tiny{2}};
						\node at (50mm,10mm) {\tiny{8}};
						\node at (50mm,6mm) {\tiny{5}};
						\node at (50mm,2mm) {\tiny{7}};
						%
						%row number
						\node at (2mm,-2mm) {\tiny{1}};
						\node at (6mm,-2mm) {\tiny{2}};
						\node at (10mm,-2mm) {\tiny{3}};
						\node at (14mm,-2mm)  {\tiny{4}};
						\node at (18mm,-2mm) {\tiny{5}};
						\node at (22mm,-2mm) {\tiny{6}};
						\node at (26mm,-2mm) {\tiny{7}};
						\node at (30mm,-2mm) {\tiny{8}};
						\node at (34mm,-2mm) {\tiny{9}};
						\node at (38mm,-2mm) {\tiny{10}};
						\node at (42mm,-2mm) {\tiny{11}};
						\node at (46mm,-2mm) {\tiny{12}};
						\filldraw[black] (6mm,14mm) circle (1.5pt);
						\filldraw[black] (10mm,18mm) circle (1.5pt);
						\filldraw[black] (14mm,22mm) circle (1.5pt);
						\filldraw[black] (18mm,26mm) circle (1.5pt);
						\filldraw[black] (22mm,30mm) circle (1.5pt);
						\filldraw[black] (26mm,34mm) circle (1.5pt);
						\filldraw[black] (30mm,38mm) circle (1.5pt);
						\filldraw[black] (18mm,10mm) circle (1.5pt);
						\filldraw[black] (22mm,22mm) circle (1.5pt);
						\filldraw[black] (26mm,26mm) circle (1.5pt);
						\filldraw[black] (30mm,30mm) circle (1.5pt);
						\filldraw[black] (34mm,34mm) circle (1.5pt);
                        \filldraw[black] (26mm,10mm) circle (1.5pt);
						\filldraw[black] (30mm,18mm) circle (1.5pt);
						\filldraw[black] (34mm,22mm) circle (1.5pt);
						\filldraw[black] (38mm,26mm) circle (1.5pt);
						\filldraw[black] (42mm,30mm) circle (1.5pt);
						\node at (24mm,-7mm) {(e)};
						
					\end{tikzpicture}
					\quad
					
					%%%%%%%%%%%%%%%%%%%%%6
					
					\begin{tikzpicture}[scale = 0.9]
						
						\def\rectanglepath{[fill=gray!20!white,draw=black!20!black]-- +(4mm,0mm) -- +(4mm,4mm) -- +(0mm,4mm) -- cycle}
						
						\def\squarepath{-- +(48mm,0mm) -- +(48mm,48mm) -- +(0mm,48mm) -- cycle}
						
						\draw (0mm,0mm)\squarepath;
						\draw [step=4mm,dotted] (0mm,0mm) grid (48mm,48mm);
						\draw (0mm,44mm)\rectanglepath;
						\draw (0mm,40mm)\rectanglepath;
						\draw (0mm,36mm)\rectanglepath;
						\draw (0mm,32mm)\rectanglepath;
						\draw (0mm,28mm)\rectanglepath;
						\draw (0mm,24mm)\rectanglepath;
						\draw (0mm,20mm)\rectanglepath;
						\draw (0mm,16mm)\rectanglepath;
						
						\draw (4mm,44mm)\rectanglepath;
						\draw (4mm,40mm)\rectanglepath;
						\draw (4mm,36mm)\rectanglepath;
						\draw (4mm,32mm)\rectanglepath;
						\draw (4mm,28mm)\rectanglepath;
						\draw (4mm,20mm)\rectanglepath;
						\draw (4mm,24mm)\rectanglepath;
						%\draw (4mm,12mm)\rectanglepath;
						
						\draw (8mm,44mm)\rectanglepath;
						\draw (8mm,40mm)\rectanglepath;
						\draw (8mm,36mm)\rectanglepath;
						\draw (8mm,32mm)\rectanglepath;
						\draw (8mm,24mm)\rectanglepath;
						\draw (8mm,28mm)\rectanglepath;
						%\draw (8mm,16mm)\rectanglepath;
						
						\draw (12mm,44mm)\rectanglepath;
						\draw (12mm,40mm)\rectanglepath;
						\draw (12mm,36mm)\rectanglepath;
						\draw (12mm,28mm)\rectanglepath;
						\draw (12mm,32mm)\rectanglepath;
						%\draw (12mm,20mm)\rectanglepath;
						
						\draw (16mm,44mm)\rectanglepath;
						\draw (16mm,40mm)\rectanglepath;
						\draw (16mm,32mm)\rectanglepath;
						\draw (16mm,36mm)\rectanglepath;
						%\draw (16mm,24mm)\rectanglepath;
						%\draw (16mm,8mm)\rectanglepath;
						
						\draw (20mm,44mm)\rectanglepath;
						\draw (20mm,36mm)\rectanglepath;
						\draw (20mm,40mm)\rectanglepath;
						%\draw (20mm,28mm)\rectanglepath;
						%\draw (20mm,16mm)\rectanglepath;
						
						\draw (24mm,44mm)\rectanglepath;
						\draw (24mm,40mm)\rectanglepath;
						%\draw (24mm,32mm)\rectanglepath;
						%\draw (24mm,20mm)\rectanglepath;
						%\draw (24mm,8mm)\rectanglepath;
						
						\draw (28mm,44mm)\rectanglepath;
						%\draw (28mm,36mm)\rectanglepath;
						%\draw (28mm,24mm)\rectanglepath;
						%\draw (28mm,16mm)\rectanglepath;
						
						\draw (32mm,44mm)\rectanglepath;
						%\draw (32mm,28mm)\rectanglepath;
						%\draw (32mm,20mm)\rectanglepath;
						
						\draw (36mm,44mm)\rectanglepath;
						%\draw (36mm,24mm)\rectanglepath;
						
						%\draw (40mm,28mm)\rectanglepath;
						
						\draw[thick](2mm,0mm)--(2mm,14mm)--(6mm,14mm)--(6mm,18mm)--(10mm,18mm)--(10mm,22mm)--(14mm,22mm)--(14mm,26mm)--(18mm,26mm)--(18mm,30mm)--(22mm,30mm)--(22mm,34mm)--(26mm,34mm)--(26mm,38mm)--(30mm,38mm)--(30mm,42mm)--(48mm,42mm);
						\draw[thick](6mm,0mm)--(6mm,10mm)--(18mm,10mm)--(18mm,14mm)--(48mm,14mm);
						\draw[thick](10mm,0mm)--(10mm,14mm)--(18mm,14mm)--(18mm,18mm)--(22mm,18mm)--(22mm,22mm)--(26mm,22mm)--(26mm,22mm)--(26mm,26mm)--(30mm,26mm)--(30mm,30mm)--(34mm,30mm)--(34mm,34mm)--(48mm,34mm);
						\draw[thick](14mm,0mm)--(14mm,18mm)--(18mm,18mm)--(18mm,22mm)--(22mm,22mm)--(22mm,26mm)--(26mm,26mm)--(26mm,30mm)--(30mm,30mm)--(30mm,34mm)--(34mm,34mm)--(34mm,38mm)--(48mm,38mm);
						\draw[thick](18mm,0mm)--(18mm,6mm)--(48mm,6mm);
						\draw[thick](22mm,0mm)--(22mm,10mm)--(26mm,10mm)--(26mm,18mm)--(30mm,18mm)--(30mm,22mm)--(34mm,22mm)--(34mm,26mm)--(38mm,26mm)--(38mm,30mm)--(42mm,30mm)--(42mm,46mm)--(48mm,46mm);
						\draw[thick](26mm,0mm)--(26mm,2mm)--(48mm,2mm);
						\draw[thick](30mm,0mm)--(30mm,10mm)--(48mm,10mm);
						\draw[thick](34mm,0mm)--(34mm,18mm)--(48mm,18mm);
						\draw[thick](38mm,0mm)--(38mm,22mm)--(48mm,22mm);
						\draw[thick](42mm,0mm)--(42mm,26mm)--(48mm,26mm);
						\draw[thick](46mm,0mm)--(46mm,30mm)--(48mm,30mm);

						%column number
						\node at (50mm,46mm) {\tiny{6}};
						\node at (50mm,42mm) {\tiny{1}};
						\node at (50mm,38mm) {\tiny{4}};
						\node at (50mm,34mm) {\tiny{3}};
						\node at (50mm,30mm) {\tiny{12}};
						\node at (50mm,26mm) {\tiny{11}};
						\node at (50mm,22mm) {\tiny{10}};
						\node at (50mm,18mm) {\tiny{9}};
						\node at (50mm,14mm) {\tiny{2}};
						\node at (50mm,10mm) {\tiny{8}};
						\node at (50mm,6mm) {\tiny{5}};
						\node at (50mm,2mm) {\tiny{7}};
						%
						%row number
						\node at (2mm,-2mm) {\tiny{1}};
						\node at (6mm,-2mm) {\tiny{2}};
						\node at (10mm,-2mm) {\tiny{3}};
						\node at (14mm,-2mm)  {\tiny{4}};
						\node at (18mm,-2mm) {\tiny{5}};
						\node at (22mm,-2mm) {\tiny{6}};
						\node at (26mm,-2mm) {\tiny{7}};
						\node at (30mm,-2mm) {\tiny{8}};
						\node at (34mm,-2mm) {\tiny{9}};
						\node at (38mm,-2mm) {\tiny{10}};
						\node at (42mm,-2mm) {\tiny{11}};
						\node at (46mm,-2mm) {\tiny{12}};
						\filldraw[black] (6mm,14mm) circle (1.5pt);
						\filldraw[black] (10mm,18mm) circle (1.5pt);
						\filldraw[black] (14mm,22mm) circle (1.5pt);
						\filldraw[black] (18mm,26mm) circle (1.5pt);
						\filldraw[black] (22mm,30mm) circle (1.5pt);
						\filldraw[black] (26mm,34mm) circle (1.5pt);
						\filldraw[black] (30mm,38mm) circle (1.5pt);
						\filldraw[black] (18mm,10mm) circle (1.5pt);
						\filldraw[black] (22mm,18mm) circle (1.5pt);
						\filldraw[black] (26mm,22mm) circle (1.5pt);
						\filldraw[black] (30mm,26mm) circle (1.5pt);
						\filldraw[black] (34mm,30mm) circle (1.5pt);
						
						\filldraw[black] (26mm,10mm) circle (1.5pt);
						\filldraw[black] (30mm,18mm) circle (1.5pt);
						\filldraw[black] (34mm,22mm) circle (1.5pt);
						\filldraw[black] (38mm,26mm) circle (1.5pt);
						\filldraw[black] (42mm,30mm) circle (1.5pt);
						\node at (24mm,-7mm) {(f)};
						
					\end{tikzpicture}
				\end{tabular}
				\caption{The algorithm  on the snow Rothe pipedream of $w=[6,1,4,3,12,11,10,9,2,8,5,7]$.}\label{ex_6143}
			\end{center}
		\end{figure}
\begin{example}
        In Figure \ref{ex_6143}, starting from the snow Rothe pipedream $\SRPD(w)$ of $w=[6,1,4,3,12,11,\\
        10,9,2,8,5,7]$ (Figure \ref{ex_6143}(a)), we apply the algorithm sequentially to pipes 2, 3, 4, and 1 to obtain Figure \ref{ex_6143}(b). The operations on these pipes are omitted here because they involve only the ``Drooping'' algorithm; thus, the transition from (a) to (b) is achieved directly.
		Next, applying the ``Drooping" algorithm to pipe 6 transforms (b) into (c). This operation, in turn, sequentially triggers the ``Undrooping" algorithm on pipes 1, 4, and 3, leading from (c) through to (f). The red line marks the pipe that will execute the ``Undrooping" algorithm and the thick segment represents a long line. 
	\end{example}

    We end this section with an example to show that this algorithm is not the composition of constructions defined by Chou and Yu~\cite{cy24} and Huang, Shimozono, and Yu~\cite{HSY24}.
\begin{example}
    \label{e:hsy}
    For $w=21453$, the left marked bumpless pipedream is the unique maximal marked bumpless pipedream constructed from our algorithm. The right pipedream is obtained by applying the inverse of the bijection in~\cite{HSY24} to the left marked bumpless pipedream. The right pipedream is not the unique maximal pipedream constructed in~\cite{cy24} since its column weight is not $(2,1,2,0,0)$.
		\begin{center}
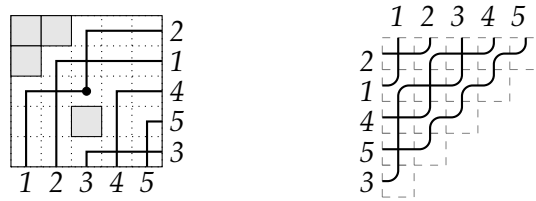
		\begin{tabular}{cc}
          \begin{tikzpicture}[scale = 1]
        							
        							\def\rectanglepath{[fill=gray!20!white,draw=black!20!black]-- +(4mm,0mm) -- +(4mm,4mm) -- +(0mm,4mm) -- cycle}
        							
        							\def\squarepath{-- +(20mm,0mm) -- +(20mm,20mm) -- +(0mm,20mm) -- cycle}

        							\draw (0mm,0mm)\squarepath;
        							\draw [step=4mm,dotted] (0mm,0mm) grid (20mm,20mm);
        							\draw (0mm,16mm) \rectanglepath;
        \draw (0mm,12mm) \rectanglepath;
        \draw (4mm,16mm) \rectanglepath;
        \draw (8mm,4mm) \rectanglepath;
        \filldraw[black] (10mm,10mm) circle (1.5pt);
        							\draw[thick](2mm,0mm)--(2mm,10mm)--(10mm,10mm)--(10mm,14mm)--(20mm,14mm);
        							\draw[thick](6mm,0mm)--(6mm,14mm)--(10mm,14mm)--(10mm,18mm)--(20mm,18mm);
        							\draw[thick](10mm,0mm)--(10mm,2mm)--(20mm,2mm);
        							\draw[thick](14mm,0mm)--(14mm,10mm)--(20mm,10mm);
        							\draw[thick](18mm,0mm)--(18mm,6mm)--(20mm,6mm);					
        %column number
        \node at (22mm,18mm) {\small{2}};
        \node at (22mm,14mm) {\small{1}};
        \node at (22mm,10mm) {\small{4}};
        \node at (22mm,6mm) {\small{5}};
        \node at (22mm,2mm) {\small{3}};
        							
        %row number
        \node at (2mm,-2mm) {\small{1}};
        \node at (6mm,-2mm) {\small{2}};
        \node at (10mm,-2mm) {\small{3}};
        \node at (14mm,-2mm)  {\small{4}};
        \node at (18mm,-2mm) {\small{5}};
        %\node at (12mm,-8mm) {\small{(a)}};
        \end{tikzpicture}&
        \quad \quad \quad \quad
        \begin{tikzpicture}[x=1em,y=1em,thick,rounded corners, color = black]
        \draw[step=1,gray,ultra thin,dashed] (0,0) grid (1,1);
        \draw[step=1,gray,ultra thin,dashed] (0,1) grid (2,2);
        \draw[step=1,gray,ultra thin,dashed] (0,2) grid (3,3);
        \draw[step=1,gray,ultra thin,dashed] (0,3) grid (4,4);
        \draw[step=1,gray,ultra thin,dashed] (0,4) grid (5,5);
        \draw (0,0.5)--(0.5, 0.5)--(0.5, 3.5)--(2.5, 3.5)--(2.5, 5);
        \draw (0,1.5)--(1.5, 1.5)--(1.5, 2.5)--(2.5, 2.5)--(2.5, 3.5)--(3.5, 3.5)--(3.5, 4.5)--(4.5, 4.5)--(4.5, 5);
        \draw (0,2.5)--(1.5,2.5)--(1.5,4.5)--(3.5,4.5)--(3.5,5);
        \draw (0,3.5)--(0.5,3.5)--(0.5,5);
        \draw (0,4.5)--(1.5, 4.5)--(1.5, 5);
        %column number
        \node at (-2mm,18mm) {\small{2}};
        \node at (-2mm,14mm) {\small{1}};
        \node at (-2mm,10mm) {\small{4}};
        \node at (-2mm,6mm) {\small{5}};
        \node at (-2mm,2mm) {\small{3}};
        							
        %row number
        \node at (2mm,24mm) {\small{1}};
        \node at (6mm,24mm) {\small{2}};
        \node at (10mm,24mm) {\small{3}};
        \node at (14mm,24mm)  {\small{4}};
        \node at (18mm,24mm) {\small{5}};
        %\node at (12mm,-4mm) {\small{(b)}};
        \end{tikzpicture}
        \end{tabular}
		\end{center}
\end{example}
	
	\section{Proof of Theorem~\ref{T: Main Theorem}}
	\label{S: Proof of main theorem}
	
	We now prove that the algorithm produces the desired marked bumpless pipedream. We call the combined step of moving left and then downward a \defi{bi-step}. There are four bi-steps in Figure \ref{bi}, one of which is marked by a red line. Let $num_{*}(p)$ be the number of {\large{$*$}} on pipe $p$.
    
	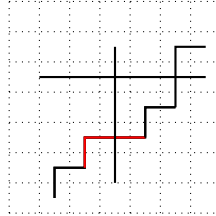
\begin{figure}[H]
			\begin{tikzpicture}[scale = 1]
				\def\rectanglepath{[fill=gray!20!white,draw=black!20!black]-- +(4mm,0mm) -- +(4mm,4mm) -- +(0mm,4mm) -- cycle}
				
				\draw [step=4mm,dotted] (-4mm,-4mm) grid (24mm,24mm);
				\draw[thick](2mm,-2mm)--(2mm,2mm)--(6mm,2mm)--(6mm,6mm)--(14mm,6mm)--(14mm,10mm)--(18mm,10mm);
				\draw [thick](18mm,10mm)--(18mm,18mm)--(22mm,18mm);
				\draw[thick](2mm,-2mm)--(2mm,2mm)--(6mm,2mm)--(6mm,6mm)--(14mm,6mm);
				\draw[thick][red](6mm,2mm)--(6mm,6mm)--(14mm,6mm);
				\draw[thick](14mm,6mm)--(14mm,10mm)--(18mm,10mm);
				
				\draw [thick](0mm,14mm)--(22mm,14mm);
				\draw [thick](10mm,0mm)--(10mm,18mm);
			\end{tikzpicture}
			\caption{Illustration of a bi-step: the left-then-down movement of a pipe during the algorithm.}\label{bi}
	\end{figure}
    
	\begin{lem}\label{l1}
		(1) The number of new empty boxes generated by the ``Drooping" algorithm in the starting row of pipe $p$ is equal to the number of {\large{$*$}} on pipe $p$.
		
		(2) Assume that after pipe $p$ completes the ``Drooping" algorithm, there are no long lines remaining. Right after pipe $p$ completes the algorithm, the number of empty boxes added on its left side decreases by one in each successive row that contains a $\rtile$ of $p$, until it reaches zero (see Figure \ref{emptybox}). This holds even if other pipes pass through pipe $p$ during the process.
	\end{lem}
    
	\begin{figure}[h]
		\begin{center}
			\begin{tabular}{cc}

				\begin{tikzpicture}[scale = 1]

					\draw [step=4mm,dotted] (-4mm,-4mm) grid (24mm,24mm);

					\node at (18mm,18mm) {{\large $*$} };
					\node at (14mm,18mm) {{\large $*$} };
					\node at (10mm,18mm) {{\large $*$} };
					
					\draw[thick](2mm,0mm)--(2mm,18mm)--(20mm,18mm);

				\end{tikzpicture}
				
				&
				
				\quad
				
				\begin{tikzpicture}[scale = 1]
					
					\def\rectanglepath{[fill=gray!20!white,draw=black!20!black]-- +(4mm,0mm) -- +(4mm,4mm) -- +(0mm,4mm) -- cycle}

					%------
					\draw [step=4mm,dotted] (-4mm,-4mm) grid (24mm,24mm);
					\draw [thin](0mm,16mm)\rectanglepath;
					\draw [thin](0mm,12mm)\rectanglepath;
					\draw [thin](0mm,8mm)\rectanglepath;
					\draw [thin](4mm,16mm)\rectanglepath;
					\draw [thin](4mm,12mm)\rectanglepath;
					\draw [thin](8mm,16mm)\rectanglepath;

					\draw[thick](2mm,0mm)--(2mm,6mm)--(6mm,6mm)--(6mm,10mm)--(10mm,10mm)--(10mm,14mm)--(14mm,14mm)--(14mm,18mm)--(20mm,18mm);

				\end{tikzpicture}
				
			\end{tabular}
			\caption{Newly added empty boxes on the left side of the pipe during ``Drooping".}\label{emptybox}
		\end{center}
	\end{figure}
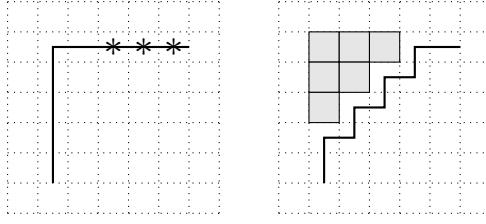
	\begin{proof}
		(1) It is clear from the algorithm. 
		
		(2) Suppose $num_{*}(p) = t$. At the redrawing stage of pipe \(p\), once the second \(\rtile\) on pipe \(p\) is generated, the number of \(\btile\) newly added to the left of this \(\rtile\) in that row is \(t-1\). 
		This pattern continues, the number of $\btile$ newly added to the left side of pipe $p$ decreases by one in each subsequent row.
		
		It can be observed that if other pipes pass through pipe $p$ horizontally or vertically, it does not affect the result. This is because when pipe $p$ encounters $\htile$ or $\vtile$, the algorithm's rule is to pass through directly. Therefore, by the definition of a bi-step, the conclusion still holds.
	\end{proof}
	
	\begin{defn}
		For a row $i$ containing two distinct pipes, denote the more leftward pipe as $p$ and the more rightward pipe as $q$. Let $j_1$ be the largest column index such that $(i, j_1)$ contains pipe $p$, and $j_2$ be the smallest column index such that $(i, j_2)$ contains pipe $q$.
		
		The \defi{gap} between the two pipes in row $i$ is defined as $j_2 - j_1$.
	\end{defn}
	
	\begin{lem}\label{l2}
        If two pipes do not cross in $\RPD(w)$, then they do not cross in any intermediate step of the algorithm.
	\end{lem}
	\begin{proof}
        Let pipe $p,q$ be a pair of non-crossing pipes such that the starting row of pipe $p$ is below the starting row of pipe $q$, i.e. $w^{-1}(p) > w^{-1}(q)$. Denote the terminating columns of pipes $q$ and $p$ as $j_1$ and $j_2$, respectively and their starting rows as $i_1$ and $i_2$. So we have that $j_1 = q$, $j_2 = p$, $i_1 = w^{-1}(q)$ and $i_2 = w^{-1}(p)$. 
		
		{\bf Case 1.}
		First, we discuss the simplest scenario: consider a rectangle whose NW corner is at $(i_1, j_1)$ and SE corner is at $(i_1 + num_*(p)+l, j_1 + num_*(q))$, and assume that before pipe $p$ executes the ``Drooping" algorithm, there are no other pipes besides $p$ and $q$ within this rectangle. Suppose in this rectangle there are $l$ empty rows and $k$ empty columns between the two pipes, where $k$ and $l$ can be any positive integers. 
		
		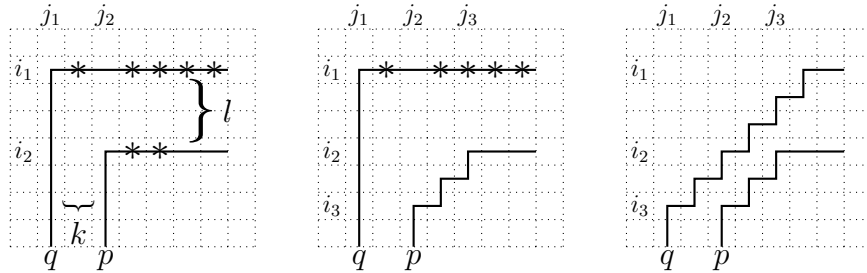
\begin{figure}[h]
			\begin{center}
				\begin{tabular}{cccc}
					
					%%%1
					
					\begin{tikzpicture}[scale = 0.9]

						%-------
						\draw [step=4mm,dotted] (-12mm,-8mm) grid (24mm,24mm);
						
						\draw[thick](-6mm,-8mm)--(-6mm,18mm)--(20mm,18mm);
						\draw[thick](2mm,-8mm)--(2mm,6mm)--(20mm,6mm);
						
						\node at (-2mm,18mm) {{\large $*$} };
						\node at (18mm,18mm) {{\large $*$} };
						\node at (6mm,18mm) {{\large $*$} };
						\node at (10mm,18mm) {{\large $*$} };
						\node at (14mm,18mm) {{\large $*$} };
						\node at (6mm,6mm) {{\large $*$} };
						\node at (10mm,6mm) {{\large $*$} };

						\node at (-6mm,-10mm) {$q$};
						\node at (2mm,-10mm) {$p$};
						\node[scale=0.8] at (-6mm,26mm) {$j_1$};
						\node[scale=0.8] at (2mm,26mm) {$j_2$};
						\node[scale=0.8] at (-10mm,18mm) {$i_1$};
						\node[scale=0.8] at (-10mm,6mm) {$i_2$};
						\node at (20mm,12mm) {$l$};
						\node [scale=2] at (16mm,12mm) {$\}$};
						
						\node at (-2mm,-6mm) {$k$};
						\node[rotate = 270] at (-2mm,-2mm) {$\}$};
						
					\end{tikzpicture}
					&
					
					\quad
					%%%2
					\begin{tikzpicture}[scale = 0.9]

						%-------
						\draw [step=4mm,dotted] (-12mm,-8mm) grid (24mm,24mm);
						
						\draw[thick](-6mm,-8mm)--(-6mm,18mm)--(20mm,18mm);
						\draw[thick](2mm,-8mm)--(2mm,-2mm)--(6mm,-2mm)--(6mm,2mm)--(10mm,2mm)--(10mm,6mm)--(20mm,6mm);
						
						\node at (-2mm,18mm) {{\large $*$} };
						\node at (18mm,18mm) {{\large $*$} };
						\node at (6mm,18mm) {{\large $*$} };
						\node at (10mm,18mm) {{\large $*$} };
						\node at (14mm,18mm) {{\large $*$} };
						\node[scale=0.8] at (10mm,26mm) {$j_3$};
						\node[scale=0.8] at (-6mm,26mm) {$j_1$};
						\node[scale=0.8] at (2mm,26mm) {$j_2$};
						\node[scale=0.8] at (-10mm,18mm) {$i_1$};
						\node[scale=0.8] at (-10mm,6mm) {$i_2$};
						\node[scale=0.8] at (-10mm,-2mm) {$i_3$};
						
						\node at (-6mm,-10mm) {$q$};
						\node at (2mm,-10mm) {$p$};
						
					\end{tikzpicture}
					&
					
					\quad
					%%%3
					\begin{tikzpicture}[scale = 0.9]

						%-------
						\draw [step=4mm,dotted] (-12mm,-8mm) grid (24mm,24mm);
						
						\draw[thick](-6mm,-8mm)--(-6mm,-2mm)--(-2mm,-2mm)--(-2mm,2mm)--(2mm,2mm)--(2mm,6mm)--(6mm,6mm)--(6mm,10mm)--(10mm,10mm)--(10mm,14mm)--(14mm,14mm)--(14mm,18mm)--(20mm,18mm);
						\draw[thick](2mm,-8mm)--(2mm,-2mm)--(6mm,-2mm)--(6mm,2mm)--(10mm,2mm)--(10mm,6mm)--(20mm,6mm);
						
						\node[scale=0.8] at (10mm,26mm) {$j_3$};
						\node[scale=0.8] at (-6mm,26mm) {$j_1$};
						\node[scale=0.8] at (2mm,26mm) {$j_2$};
						\node[scale=0.8] at (-10mm,18mm) {$i_1$};
						\node[scale=0.8] at (-10mm,6mm) {$i_2$};
						\node[scale=0.8] at (-10mm,-2mm) {$i_3$};
						
						\node at (-6mm,-10mm) {$q$};
						\node at (2mm,-10mm) {$p$};
						
					\end{tikzpicture}
					
				\end{tabular}
				\caption{Positional relationship between two pipes $p$ and $q$ in the rectangle before Drooping.}
			\end{center}
		\end{figure}
		
		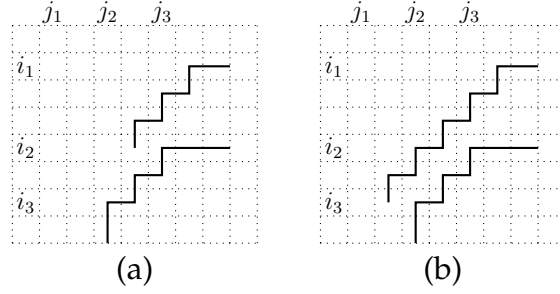
\begin{figure}[h]
			\begin{center}
				\begin{tabular}{cc}
					\begin{tikzpicture}[scale = 0.9]
						
						\draw [step=4mm,dotted] (-12mm,-8mm) grid (24mm,24mm);
						
						\draw[thick](6mm,6mm)--(6mm,10mm)--(10mm,10mm)--(10mm,14mm)--(14mm,14mm)--(14mm,18mm)--(20mm,18mm);
						\draw[thick](2mm,-8mm)--(2mm,-2mm)--(6mm,-2mm)--(6mm,2mm)--(10mm,2mm)--(10mm,6mm)--(20mm,6mm);
						
						\node[scale=0.8] at (10mm,26mm) {$j_3$};
						\node[scale=0.8] at (-6mm,26mm) {$j_1$};
						\node[scale=0.8] at (2mm,26mm) {$j_2$};
						\node[scale=0.8] at (-10mm,18mm) {$i_1$};
						\node[scale=0.8] at (-10mm,6mm) {$i_2$};
						\node[scale=0.8] at (-10mm,-2mm) {$i_3$};
						\node at (6mm,-12mm){(a)};
						
					\end{tikzpicture}
					&
					\quad
					\begin{tikzpicture}[scale = 0.9]
						%-------
						\draw [step=4mm,dotted] (-12mm,-8mm) grid (24mm,24mm);
						
						\draw[thick](-2mm,-2mm)--(-2mm,2mm)--(2mm,2mm)--(2mm,6mm)--(6mm,6mm)--(6mm,10mm)--(10mm,10mm)--(10mm,14mm)--(14mm,14mm)--(14mm,18mm)--(20mm,18mm);
						\draw[thick](2mm,-8mm)--(2mm,-2mm)--(6mm,-2mm)--(6mm,2mm)--(10mm,2mm)--(10mm,6mm)--(20mm,6mm);
						
						\node[scale=0.8] at (10mm,26mm) {$j_3$};
						\node[scale=0.8] at (-6mm,26mm) {$j_1$};
						\node[scale=0.8] at (2mm,26mm) {$j_2$};
						\node[scale=0.8] at (-10mm,18mm) {$i_1$};
						\node[scale=0.8] at (-10mm,6mm) {$i_2$};
						\node[scale=0.8] at (-10mm,-2mm) {$i_3$};
						\node at (6mm,-12mm){(b)};
					\end{tikzpicture}
					
				\end{tabular}
				\caption{Intermediate stages during the ``Drooping" of pipe $q$, illustrating the gap calculation.}\label{drawnew}
			\end{center}
		\end{figure}
		
		After applying the ``Drooping" algorithm on pipe $p$, denote the row containing the $\rtile$ of pipe $p$ in column $j_2$ as $i_3$ and the column containing $\rtile$ of pipe $p$ in row $i_2$ as $j_3$; formally, $i_3 = i_2 + num_*(p), j_3 = j_2 + num_*(p)$.

    We first consider the extremal case in which pipe \(q\) has the maximum possible number of {\large{$*$}}. That is, below row $i_1$ and between columns $j_1$ and $j_2$, there are $k$ dark clouds; to the right of column $j_1$ and between rows $i_1$ and $i_2$, there are $l$ dark clouds. Assuming $num_{*}(p)=t$ then $num_{*}(q)=t+k+l$.

		In row $i_2$, pipe $p$ begins its first bi-step by placing a $\rtile$ at $(i_2,j_3)$, at which point $j_3=j_1+t + k+1$.
		On the other hand, pipe $q$ needs $l+1$ bi-steps to reach the $i_2$-th row. 
        By Lemma \ref{l1}, \(k+t\) new empty boxes appear to the left of pipe \(q\) in row \(i_2\), as shown in Figure \ref{drawnew} (a). 
		At this point, in row $i_2$, $\rtile$ of pipe $p$ is in column $j_3 = j_1 + k +t+1 $, $\mtile$ of pipe $q$ is in column $ j_1 + k +t $. So, the gap between pipe $p$ and $q$ in row $i_2$ is $1$. At the same time, pipe $q$ requires $t+l+1$ bi-steps to reach row $i_3$. By Lemma \ref{l1}, upon reaching this row, pipe \(q\) gains \(k\) new empty boxes on its left side, as shown in Figure \ref{drawnew}(b). 
		In row $i_3$, $\rtile$ of pipe $p$ is in column $j_2 = j_1 +k+1$, $\mtile$ of pipe $q$ is in column $ j_1 + k $. So, the gap between pipe $p$ and $q$ in row $i_3$ is $1$. By Lemma \ref{l1}, they also do not cross between row $i_2$ and $i_3$.

        If $ num_{*}(p)+k+l > num_{*}(q) $, then the horizontal gap between \(p\) and \(q\) in each relevant row is at least as large as in the extremal case considered above. In particular, the gap is always at least \(1\), and hence the two pipes do not cross.

		{\bf Case 2.} In this case, there exists another pipe $r$ inside the rectangle whose NW corner is $(i_1, j_1)$ and SE corner is $(i_1 + l + num_*(p), j_1 + num_*(q))$.
		
		\begin{itemize}
			
			\item[(1)] If $w^{-1}(q) < w^{-1}(r) < w^{-1}(p)$,  at this point, the starting row of pipe $r$  lies between the starting rows of pipes $p$ and $q$.
			
			\begin{itemize}
				\item[$\bullet$] If $q < r < p$, reasoning analogous to Case 1 shows that $r$ does not cross $p$ and $q$ does not cross $r$. Therefore, $p$ and $q$ are also disjoint, so the conclusion remains valid.
				
				\item[$\bullet$] If $r < q < p$ and there is a p-cross between pipe \(q\) and \(r\) in column \(q\) after the “Drooping” algorithm on pipe \(q\) is applied. Compared with Case 1, one of the empty rows between pipe $p$ and $q$ is occupied by pipe $r$.
				In this situation, to the right of column $j_1$ and between rows $i_1$ and $i_2$, the maximum possible number of dark clouds is $l-1$. Hence the maximum possible value of $num_*(q)$ decreases by 1.  It follows from Case 1 that pipes $p$ and $q$ still do not cross.
				
				If $r < q < p$ and there is no $\ptile$ between pipe \(q\) and \(r\) in column \(q\) after the ``Drooping" algorithm on pipe $q$ is applied, as shown in Figure \ref{10}. According to the algorithm, pipe $q$ can only pass through pipe $r$ in the starting row of $r$, that is, row $w^{-1}(r)$. This may cause pipe $p$ to cross pipe $q$.
				Based on the Case 1, one of the empty rows between pipe $p$ and $q$ is occupied by pipe $r$, resulting in the maximum possible value of $num_*(q)$ being reduced by $1$ compared to Case 1.
				We now consider the situation that pipe $q$ has the most {\large{$*$}}, that is , $num_*(q)=t+k+l-1$.
				When pipe $q$ passes through pipe $r$, it proceeds one row farther downward. Pipe $q$ requires $l$ bi-steps to reach row $i_2$, at which point it is located in column $j_1+t+k$. Similarly, when there are \( c \) pipes analogous to \( r \), the same conclusion holds. And pipe $p$ in row $i_2$ is in column $j_3 = j_1+t+k+1$. So the relative positions of pipes $q$ and $p$ remain the same as in Case 1. Equivalently, each additional downward passage through such a pipe reduces the maximum possible value of $num_*(q)$ by exactly one. Abstractly, this amounts to the number of extra tiles it traverses when passing downward through pipes such as \(r\) being ``offset" by the number of {\large{$*$}} that pipe \(q\) lacks. Therefore the gap between \(p\) and \(q\) remains at least \(1\), the conclusion still holds. 
                
				\begin{figure}[h]
					\begin{center}
						\begin{tabular}{cc}
							
							%%%1
							
							\begin{tikzpicture}[scale = 0.9]
								
								%-------
								\draw [step=4mm,dotted] (-14mm,-8mm) grid (24mm,24mm);
								
								\draw[thick](-6mm,-8mm)--(-6mm,18mm)--(20mm,18mm);
								\draw[thick](2mm,-8mm)--(2mm,6mm)--(20mm,6mm);
								\draw[thick](-14mm,-8mm)--(-14mm,14mm)--(20mm,14mm);
								
								\node at (-6mm,-10mm) {$q$};
								\node at (2mm,-10mm) {$p$};
								\node at (-14mm,-10mm) {$r$};
								\node[scale=0.8] at (-6mm,26mm) {$j_1$};
								\node[scale=0.8] at (2mm,26mm) {$j_2$};
								\node[scale=0.8] at (-18mm,18mm) {$i_1$};
								\node[scale=0.8] at (-18mm,6mm) {$i_2$};

							\end{tikzpicture}
							&
							
							\quad
							%%%2
							\begin{tikzpicture}[scale = 0.9]
								
								%-------
								\draw [step=4mm,dotted] (-18mm,-8mm) grid (24mm,24mm);
								
								\draw[thick](-14mm,-8mm)--(-14mm,-2mm)--(-10mm,-2mm)--(-10mm,2mm)--(-2mm,2mm)--(-2mm,6mm)--(2mm,6mm)--(2mm,10mm)--(6mm,10mm)--(6mm,14mm)--(20mm,14mm);
								\draw[thick](-6mm,-8mm)--(-6mm,-2mm)--(-2mm,-2mm)--(-2mm,2mm)--(2mm,2mm)--(2mm,6mm)--(6mm,6mm)--(6mm,10mm)--(10mm,10mm)--(10mm,18mm)--(20mm,18mm);
								\draw[thick](2mm,-8mm)--(2mm,-2mm)--(6mm,-2mm)--(6mm,2mm)--(10mm,2mm)--(10mm,6mm)--(20mm,6mm);

								\node[scale=0.8] at (10mm,26mm) {$j_3$};
								\node[scale=0.8] at (-6mm,26mm) {$j_1$};
								\node[scale=0.8] at (2mm,26mm) {$j_2$};
								\node[scale=0.8] at (-18mm,18mm) {$i_1$};
								\node[scale=0.8] at (-18mm,6mm) {$i_2$};
								\node[scale=0.8] at (-18mm,-2mm) {$i_3$};
								
								\node at (-6mm,-10mm) {$q$};
								\node at (2mm,-10mm) {$p$};
								\node at (-14mm,-10mm) {$r$};
								
							\end{tikzpicture}
							
						\end{tabular}
						\caption{Case when $r < q < p$ and there is no cross between pipe \(q\) and \(r\) in column \(q\) after the ``Drooping" algorithm on pipe $q$ is applied.}\label{10}
					\end{center}
				\end{figure}
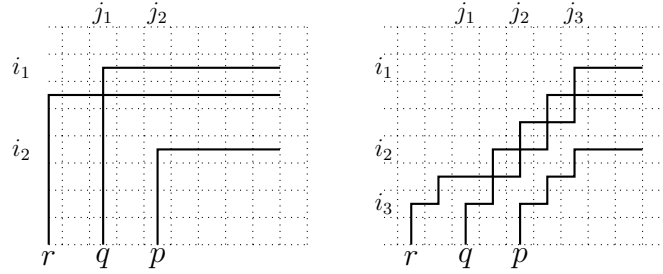
				
				\item[$\bullet$] If $q < p < r$ and there is a p-cross between pipe $r$ and pipe $p$ in column $r$ after the ``Drooping'' algorithm on pipe $r$ is applied. The same gap estimate as in the previous subcase applies, the gap between \(p\) and \(q\) remains at least \(1\). The number of extra tiles that pipe \(p\) traverses when moving leftward through pipes such as \(r\) ``offsets" the number of {\large{$*$}} that pipe \(q\) lacks. Therefore, pipes \(p\) and \(q\) remain disjoint. 
				
				If $q < p < r$ and there is no $\ptile$ between pipe $r$ and pipe $p$ in column $r$ after the ``Drooping'' algorithm on pipe $r$ is applied. According to the algorithm, pipe $r$ can only pass through pipe $p$ in the starting row of $p$, that is, row $w^{-1}(p)$. Pipe $r$ triggers ``Undrooping'' algorithm for pipe $p$.
				First, pipe $p$ executes the ``Drooping" algorithm. Since pipe $r$ passes through pipe $p$ at the starting row of $p$, its first turning point shifts one column right compared to Case 1, that is, to column $j_3 + 1$.
				So the first $\rtile$ appears at $(i_2,\; j_3 + 1)$. Then pipe $r$ performs the ``Drooping'' algorithm, 
                
                which forces pipe $p$ to undergo the ``Undrooping'' algorithm, as illustrated in Figure \ref{qpr} (b). After the undrooping, pipe $p$ returns to the position it would have had without the influence of $r$, i.e., the same position as pipe $p$ in Case 1, see Figure \ref{qpr} (c). So the earlier conclusion remains valid.
				If there are $c$ pipes analogous to $r$, the first $\rtile$ of pipe $p$ will be located at $(i_2,\; j_3 + c)$ and requires $c$ undrooping operations; ultimately it still returns to the position of pipe $p$ in Case 1.
\begin{figure}[h]
			\begin{center}
				\begin{tabular}{ccc}
					
					%%%1
					
					\begin{tikzpicture}[scale = 0.9]
						
						%-------
						\draw [step=4mm,dotted] (-12mm,-8mm) grid (24mm,24mm);
						
						\draw[thick](-6mm,-8mm)--(-6mm,18mm)--(20mm,18mm);
						\draw[thick](2mm,-8mm)--(2mm,6mm)--(20mm,6mm);
						\draw[thick](10mm,-8mm)--(10mm,10mm)--(20mm,10mm);
						
						\node at (-6mm,-10mm) {$q$};
						\node at (2mm,-10mm) {$p$};
						\node at (10mm,-10mm) {$r$};
						\node[scale=0.8] at (-6mm,26mm) {$j_1$};
						\node[scale=0.8] at (2mm,26mm) {$j_2$};
						\node[scale=0.8] at (-14mm,18mm) {$i_1$};
						\node[scale=0.8] at (-14mm,6mm) {$i_2$};
						
						\node at (6mm,-16mm) {(a)};

					\end{tikzpicture}
					&
					
					\quad
					%%%2
					\begin{tikzpicture}[scale = 0.9]
						
						%-------
						\draw [step=4mm,dotted] (-12mm,-8mm) grid (24mm,24mm);
						
						\draw[thick](-6mm,-8mm)--(-6mm,-2mm)--(-2mm,-2mm)--(-2mm,2mm)--(2mm,2mm)--(2mm,6mm)--(6mm,6mm)--(6mm,10mm)--(10mm,10mm)--(10mm,14mm)--(14mm,14mm)--(14mm,18mm)--(20mm,18mm);
						\draw[thick](10mm,-8mm)--(10mm,-2mm)--(14mm,-2mm)--(14mm,2mm)--(18mm,2mm)--(18mm,10mm)--(20mm,10mm);
						\draw[thick](2mm,-8mm)--(2mm,-2mm)--(6mm,-2mm)--(6mm,2mm)--(14mm,2mm)--(14mm,6mm)--(20mm,6mm);
						\node at (-6mm,-10mm) {$q$};
						\node at (2mm,-10mm) {$p$};
						\node at (10mm,-10mm) {$r$};			
						\node[scale=0.8] at (10mm,26mm) {$j_3$};
						\node[scale=0.8] at (-6mm,26mm) {$j_1$};
						\node[scale=0.8] at (2mm,26mm) {$j_2$};
						\node[scale=0.8] at (-14mm,18mm) {$i_1$};
						\node[scale=0.8] at (-14mm,6mm) {$i_2$};
						\node[scale=0.8] at (-14mm,-2mm) {$i_3$};
						\node at (6mm,-16mm) {(b)};
						
					\end{tikzpicture}
					&
					
					\quad
					%%%3
					\begin{tikzpicture}[scale = 0.9]
						
						%-------
						\draw [step=4mm,dotted] (-12mm,-8mm) grid (24mm,24mm);
						
						\draw[thick](-6mm,-8mm)--(-6mm,-2mm)--(-2mm,-2mm)--(-2mm,2mm)--(2mm,2mm)--(2mm,6mm)--(6mm,6mm)--(6mm,10mm)--(10mm,10mm)--(10mm,14mm)--(14mm,14mm)--(14mm,18mm)--(20mm,18mm);
						\draw[thick](10.5mm,-8mm)--(10.5mm,-2mm)--(14.5mm,-2mm)--(14.5mm,2mm)--(18.5mm,2mm)--(18.5mm,10mm)--(20mm,10mm);
						\draw[thick](2mm,-8mm)--(2mm,-2mm)--(6mm,-2mm)--(6mm,2mm)--(10mm,2mm)--(10mm,6mm)--(20mm,6mm);
						
						\node at (-6mm,-10mm) {$q$};
						\node at (2mm,-10mm) {$p$};
						\node at (10mm,-10mm) {$r$};				
						\node[scale=0.8] at (10mm,26mm) {$j_3$};
						\node[scale=0.8] at (-6mm,26mm) {$j_1$};
						\node[scale=0.8] at (2mm,26mm) {$j_2$};
						\node[scale=0.8] at (-14mm,18mm) {$i_1$};
						\node[scale=0.8] at (-14mm,6mm) {$i_2$};
						\node[scale=0.8] at (-14mm,-2mm) {$i_3$};
						\node at (6mm,-16mm) {(c)};
						
					\end{tikzpicture}
					
				\end{tabular}
				\caption{Case when $q < p < r$ and there is no cross between pipe $r$ and pipe $p$ in column $r$ after the ``Drooping'' algorithm on pipe $r$ is applied.}\label{qpr}
			\end{center}
		\end{figure}
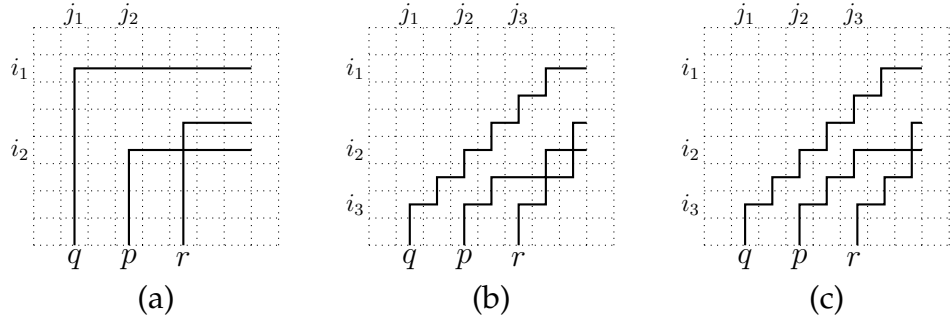
            \item[(2)] If $w^{-1}(r) < w^{-1}(q)$, at this point, the starting row of pipe $r$  lies above the starting row of pipe $q$.
			The presence of pipe $r$ may cause $q$ to extend further upward, thereby increasing its distance from $p$.
			
			\item[(3)] If $w^{-1}(r) > w^{-1}(p)$, at this point, the starting row of pipe $r$  lies below the starting row of pipe $p$.
			The presence of pipe $r$ may cause $p$ to extend further downward, thereby increasing its distance from $q$.
			
			\end{itemize}
		\end{itemize}

	\end{proof}

	\begin{rem}\label{r1}
		Suppose pipes \(p\) and \(q\) cross and $w^{-1}(p) > w^{-1}(q)$. Then during the algorithm, any p-cross generated between \(p\) and \(q\) must lie on either the starting row or the terminating column of pipe \(p\). 
	\end{rem}
	
	\begin{lem}\label{l3}
		During the algorithm, there is at most one p-cross between any two pipes. That is, the configuration shown in Figure \ref{impossible} cannot occur. 
	\end{lem}
	
	\begin{proof}
		Let pipe $p$ be the pipe whose starting row is below that of pipe $q$; i.e., $w^{-1}(p) > w^{-1}(q)$. Assume that there is more than one p-cross between pipe \(p\) and \(q\). 
		By Remark \ref{r1}, any p-cross between $p$ and $q$ must lie on either the starting row or the terminating column of $p$, see Figure \ref{impossible}. But in this configuration, pipe \(q\) terminates to the left of pipe \(p\), hence \(q<p\). Together with \(w^{-1}(q)<w^{-1}(p)\), this means that \(p\) and \(q\) do not cross in \(\RPD(w)\). By Lemma~\ref{l2}, they cannot cross during the algorithm, a contradiction.
      	
		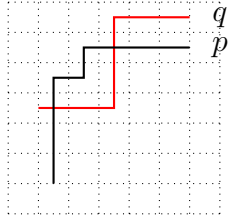
\begin{figure}[h]
			\begin{center}
				\begin{tikzpicture}[scale = 1]
					
					\def\squarepath{-- +(20mm,0mm) -- +(20mm,20mm) -- +(0mm,20mm) -- cycle}
					
					%-------
					\draw [step=4mm,dotted] (-4mm,-4mm) grid (24mm,24mm);
					
					\draw[thick,red](0mm,10mm)--(10mm,10mm)--(10mm,22mm)--(20mm,22mm);
					\draw[thick](2mm,0mm)--(2mm,14mm)--(6mm,14mm)--(6mm,18mm)--(20mm,18mm);
					
					\node at (24mm,22mm) {$q$};
					\node at (24mm,18mm) {$p$};
					
				\end{tikzpicture}
				\caption{Forbidden configuration: there is more than one p-cross between pipe \(p\) and \(q\).}\label{impossible}
				
			\end{center}
			
		\end{figure}
		
	\end{proof}

    Lemma~\ref{l3} shows that the algorithm does not create multiple p-crosses between the same pair of pipes. Thus, when the permutation is read from the resulting diagram, no additional p-cross can contribute a new inversion. Moreover, every c-cross produced by the algorithm occurs only between a pair of pipes that has already formed a p-cross; hence it does not change the set of pipe pairs that determine the permutation.

    A priori, it is unclear that after each iteration of the algorithm (finishing a combined step of drooping and undrooping of a pipe), we get a marked bumpless pipedream. We now show that the diagrams we get in each intermediate step are valid marked bumpless pipedreams. 
	\begin{prop}\label{standard}
		At each intermediate step of the algorithm, the resulting diagram is a valid marked bumpless pipedream. 
	\end{prop}
	
	\begin{figure}[h]
		\begin{center}
			\begin{tikzpicture}[scale = 1]
				
				\def\rectanglepath{-- +(16mm,0mm) -- +(16mm,12mm) -- +(0mm,12mm) -- cycle}
				\def\squarepath{[black]-- +(4mm,0mm) -- +(4mm,4mm) -- +(0mm,4mm) -- cycle}
				
				\draw [red](0mm,8mm)\rectanglepath;
				
				%-------
				\draw [step=4mm,dotted] (-4mm,-4mm) grid (24mm,24mm);
				
				\draw[thick,green](2mm,0mm)--(2mm,10mm)--(14mm,10mm)--(14mm,18mm)--(20mm,18mm);
				
				\draw (0mm,16mm)\squarepath;
				\draw (12mm,8mm)\squarepath;
				\draw[thick](6mm,6mm)--(6mm,22mm) ;
				\draw[thick](10mm,6mm)--(10mm,22mm) ;
				\draw[thick](-2mm,14mm)--(20mm,14mm) ;

				\node[scale=0.7] at (-2mm,22mm) {$(a_{i+1},b_{i-1})$};
				\node[scale=0.7] at (20mm,6mm) {$(a_{i},b_{i})$};
				
				\node at (24mm,18mm) {$q$};
				\node at (6mm,4mm) {$p_1$};
				\node at (10mm,4mm) {$p_2$};
				\node at (24mm,14mm) {$p_3$};

			\end{tikzpicture}
			\caption{Shape of tiles in the operation rectangle during mini-undroop, showing fixed tiles.}\label{f}
		\end{center}
		
	\end{figure}
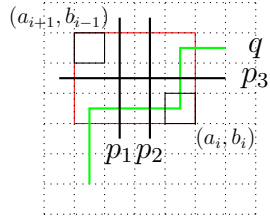
	
	\begin{proof}
		
		Throughout the algorithm, only the ``Undrooping" process may generate tiles of other shapes. Below, we discuss this case in detail. 
		
		If pipe $q$ passes through a certain tile while performing the ``Drooping" algorithm, that tile must originally contain horizontal or vertical pipes, that is, it must have been $\htile$ or $\vtile$. After the ``Drooping" algorithm on pipe $q$, these tiles become $\ptile$. Therefore, the shape of some tiles within the operation rectangle is fixed: the tiles between $b_{i-1}$-th column and $b_i$-th column in $a_i$-th row must be $\ptile$; the tiles between $a_i$-th row and $a_{i+1}$-th in $b_i$-th column must be $\ptile$. 
		
		Pipes other than $q$ in the above $\ptile$ must go directly through the operation rectangle. Otherwise, considering the paths of the pipes, some pipe would have two p-crosses with pipe $q$, which, by Lemma \ref{l3}, is impossible. For example, in the case shown in Figure \ref{f}, let the three pipes that $q$ passes through be $p_1$, $p_2$, $p_3$; they must each directly pass through the operation rectangle either horizontally or vertically. 
		
		We call the box at the NW corner of the operation rectangle the \defi{target}. For 
		example, in Figure \ref{f}, $(a_{i+1},b_{i-1})$ is the target. 
        It is clear from the algorithm that, in any given operation rectangle, if the target occurs in a \(\btile\) or \(\mtile\), it becomes a \(\rtile\) or \(\ptile\) after mini-undrooping. 
		
		Therefore, the undroop operation does not cause overlaps or conflicts between pipes; it may only change the positions of crossings involving certain pipes. The resulting diagram remains a standard marked bumpless pipedream.  
	\end{proof}
	
	\begin{prop}\label{perm remain}
		The permutation remains unchanged after the algorithm, i.e., $\widehat{\mathrm{D}}(w) \in \MBPD(w)$. 
	\end{prop}
	
	\begin{proof}
		
        By Lemma~\ref{l2}, any pair of pipes that is noncrossing in \(\RPD(w)\) remains noncrossing throughout the algorithm. Thus the algorithm creates no new crossing pair. On the other hand, the local drooping and undrooping moves only move existing crossings between already-crossing pairs, and Lemma~\ref{l3} prevents extra p-crosses from contributing new inversions when the permutation is read from the diagram. Hence the set of pipe pairs that determine the permutation is unchanged. Therefore the resulting marked bumpless pipedream still lies in \(\MBPD(w)\).
        
	\end{proof}

	\begin{prop}\label{raj}
    The marked bumpless pipedream $\widehat{\mathrm{D}}(w)$ satisfies $\rwt(\widehat{\mathrm{D}}(w))=\rajcode(w)$. 
	\end{prop}
	
	\begin{proof}
    Suppose we discuss the change in weight of row \(i\) in the intermediate diagram \(D'\) during the algorithm. 
		During the algorithm, the operation at the starting row of each pipe generates a number of new empty boxes equal to the number of {\large{$*$}}, thus leaving the $\rwt_i(D')$ unchanged, i.e. $\rwt_i(D') = \rw_i(\Snow(w))$. The effects of other operations on $\rwt_i(D')$ will be discussed later. 
		
		\textbf{Case 1.} During the drooping process, the changes in $\rwt_i(D')$ are described below. 
		\begin{itemize}
			\item[(1)] When the old pipe disappears, the following two situations may happen:
			\begin{itemize}
				\item[$\bullet$] A $\vtile$ changes to $\btile$, which increases $\rwt_i(D')$ by 1.
				\item[$\bullet$] A $\ptile$ changes to $\htile$, and if it does not lie in the starting row of the pipe contained in the $\htile$ , then it becomes a long line that requires undrooping. In this case, the current row is the bottommost row affected by the undrooping operations. According to the later analysis of Case 2 (1), this will cause $\rwt_i(D')$ to increase by 1 in this row. 
			\end{itemize}
			
			\item[(2)] When redrawing the pipe, two situations may arise:
			\begin{itemize}
				\item[$\bullet$] A $\btile$ changes to $\rtile$, or a $\mtile$ changes to $\ptile$, which decreases $\rwt_i(D')$ by 1. 
				\item[$\bullet$] A pipe passes downward through a $\htile$ and the tile turns into $\ptile$. According to the algorithm, such a $\htile$ only appears in the starting row of pipe contained in that $\htile$. 
				
				If the pipe contained in that $\htile$ later requires undrooping, then the current row is the topmost row affected by the undrooping operations. Based on the subsequent analysis of Case 2 (3), this will cause $\rwt_i(D')$ to decrease by 1 in this row. 
			\end{itemize}

		\end{itemize}
		
		When Case 1 (1) occurs, $\rwt_i(D')$ increases by 1; when Case 1 (2) occurs, $\rwt_i(D')$ decreases by 1. Thus, after the undrooping operations generated in the above process are completed, $\rwt_i(D')$ ultimately remains unchanged. 
        We now consider the exceptional case in which the above transformation does not trigger an undrooping step. In the second item of Cases 1(1) and Case 1(2), the scenario without undrooping always occurs in the same row. Specifically, in row $i$, a $\ptile$ transforms into a $\htile$ and a $\htile$ transforms into a $\ptile$; in this case, the $\rwt_i(D')$ remains unchanged. If, in row $i$, the transformation of a $\ptile$ into a $\htile$ as in Case 1(1) occurs without triggering undrooping, then the row $i$ must be the starting row of the pipe contained in that $\htile$. After pipe $p$ completes the ``Drooping" algorithm, a $\htile$ on the right side of this row transforms into a $\ptile$; hence the $\rwt_i(D')$ remains unchanged. If, in row $i$, the transformation of a $\htile$ into a $\ptile$ as in Case 1(2) occurs and the pipe passed through by $p$ subsequently undergoes no undrooping (this process occurs only at the starting row of the pipe contained in that $\htile$), then the transformation of a $\ptile$ into a $\htile$ also occurs at a starting row, and consequently the $\rwt_i(D')$ remains unchanged.
		
		\textbf{Case 2.} During the undrooping process, the changes in $\rwt_i(D')$ are described below.  Mark the rows and columns as specified in the ``Undrooping" algorithm and perform undrooping accordingly. 
		
		\begin{itemize}
			\item[(1)] When $i = a_1$, the $\htile$ changes into $\mtile$, it increases $\rwt_i(D')$ by 1. 
			
			\item[(2)] For any $i = a_2, a_3, \dots, a_r$, as can be observed from the ``Undrooping'' algorithm, the undrooping operation for a pipe is continuous: it must begin at the initial long line of the pipe and proceed until reaching the pipe's starting row. 
			
			Let \( h = 2, 3, \dots, r \). When \( i = a_h \), during the \((h-1)\)-th mini-undroop, either a $\btile$ changes to $\rtile$ or a $\mtile$ changes to $\ptile$, which decreases $\rwt_i(D')$ by 1. Then, during the \(h\)-th mini-undroop, the $\htile$ changes to $\mtile$, increasing $\rwt_i(D')$ by 1.
			
			\item[(3)] When $i = a_{r+1}$, a $\btile$ changes to $\rtile$, or a $\mtile$ changes to $\ptile$. This procedure decreases $\rwt_i(D')$ by 1. 
			
		\end{itemize}
		
		So the changes in $\rwt_i(D')$ for row $i = a_2, a_3, \dots, a_r$ cancel each other out, while the changes at the row $i = a_1$ and $i = a_{r+1}$ need to be explained in conjunction with the drooping process. Their effects have already been described in the Case 1. 

        By Definition~\ref{def}, $\rw(\Snow(w)) = \rajcode(w)$. 
        In summary, $\rwt_i(\widehat{\mathrm{D}}(w)) = \rw_i(\Snow(w))= \rajcode_i(w)$
	\end{proof}
    
    \begin{prop}\label{columnraj}
		The marked bumpless pipedream $\widehat{\mathrm{D}}(w)$ satisfies $\mathsf{cwt}(\widehat{D}(w)) =\mathsf{rajcode}(w^{-1})$. 
	\end{prop}
    \begin{proof}
        Suppose we discuss the change in weight of column \(j\) in the intermediate diagram \(D'\) during the algorithm. 
        
        We first analyze the change in $\cwt(D')$ of the terminating column caused by applying the ``Drooping" algorithm to each pipe. Suppose the pipe under consideration is \(p\), and the number of {\large{$*$}} on pipe \(p\) in the $\Snow(w)$ is \(t\) (i.e., \(\operatorname{num}_*(p) = t\)). By Lemma~\ref{l1},  the total number of bi-steps equals to $num_{*}(p)+1$. Consider the row where the \(\rtile\) of pipe \(p\) lies in its terminating column after the algorithm, set it as $i_p$. Upon reaching this tile, pipe \(p\) has completed \(t\) bi-steps and is executing the \((t+1)\)-st bi-step, having just turned downward. Hence the row \(i_p\) is equal to \(t\) plus the number of pipes that pipe \(p\) vertically passes through when being redrawn; that is, 
\begin{equation}\label{eq1}
        i_p = t + \text{ \# \{ vertically passing through pipes \} }.
        \end{equation}
        Moreover, 
        \begin{equation}\label{eq2}
            i_p = \text{ \# \{ \(\htile\) directly above the \(\rtile\) \} } + \text{ \# \{ newly generated \(\btile\) in column $p$\} }.
        \end{equation}
        By Lemma~\ref{l2}, we obtain, 
        \begin{equation}\label{eq3}
            \text{ \# \{ vertically passing through pipes \} } = \text{ \# \{ \(\htile\) directly above the \(\rtile\) \} }.
        \end{equation}
        From equation \eqref{eq1} and equation \eqref{eq3} we obtain, 
        \begin{equation}\label{eq4}
            i_p = t + \text{ \# \{ \(\htile\) directly above the \(\rtile\) \} }
        \end{equation}
        According to equation \eqref{eq2} and equation \eqref{eq4} we obtain, 
        \begin{equation}\label{eq5}
            \text{ \# \{ newly generated \(\btile\) in column $p$\} } = t.
        \end{equation}
        Moreover, from the foregoing, the number of {\large{$*$}} in column \(p\) of the $\overleftarrow{\Snow}(w)$ is also equal to \(t\). 
        Thus, equation \eqref{eq5} states that the number of {\large{$*$}} in column \(p\) of \(\overleftarrow{\Snow}(w)\) equals the number of \(\btile\) released in column \(p\) after applying the ``Drooping" algorithm. Therefore, this operation does not change the $\cwt_j(D')$, i.e. $\cwt_j(D') = \cw_j(\overleftarrow{\Snow}(w))$.

        Next, we discuss the change in $\cwt_j(D')$ for the other operations. 
        
    \textbf{Case 1.} During the drooping process, the changes in $\cwt_j(D')$ are described below. 
		\begin{itemize}
        \item[(1)] When the old pipe disappears, a $\htile$ changes to $\btile$, which increases $\cwt_j(D')$ by 1.
        \item[(2)] When redrawing the pipe, a $\btile$ changes to $\rtile$, or a $\mtile$ changes to $\ptile$, which decreases $\cwt_j(D')$ by 1. 
        \end{itemize}
        
    \textbf{Case 2.} During the undrooping process, the changes in $\cwt_j(D')$ are described below.  Mark the rows and columns as specified in the ``Undrooping" algorithm and perform undrooping accordingly. 
    \begin{itemize}
        \item[(1)] When $j=b_0,b_1,...,b_{r-1}$, the $\htile$ changes into $\mtile$, it increases $\cwt_j(D')$ by 1. Either a $\btile$ changes to $\rtile$ or a $\mtile$ changes to $\ptile$, which decreases $\cwt_j(D')$ by 1. 
        \item [(2)] When $j=b_r$, the $\mtile$ changes into $\btile$, leaving $\cwt_j(D')$ unchanged. 
    \end{itemize} 
    By Definition~\ref{def}, $\cw(\overleftarrow{\Snow}(w)) = \rajcode(w^{-1})$. Therefore, $\cwt_j(\widehat{D}(w)) =\cw_j(\overleftarrow{\Snow}(w)) = \rajcode_j(w^{-1})$
    \end{proof}

    Finally, we prove the main theorem of this paper.
    \begin{proof}[Proof of Theorem~\ref{T: Main Theorem}]
    	Given a snow Rothe pipedream $\SRPD(w)$, applying the algorithm described in this paper yields the marked bumpless pipedream $\widehat{\mathrm{D}}(w)$. By Proposition \ref{perm remain}, $\widehat{\mathrm{D}}(w) \in \MBPD(w)$. 
    	Moreover, Proposition \ref{standard} guarantees that every intermediate diagram generated during each step of the algorithm is a valid marked bumpless pipedream.
    	Finally, by Proposition \ref{raj} and Proposition \ref{columnraj}, we have $\rwt(\widehat{\mathrm{D}}(w))=\rajcode(w)$ and $\mathsf{cwt}(\widehat{D}(w)) =\mathsf{rajcode}(w^{-1})$. So the marked bumpless pipedream $\widehat{\mathrm{D}}(w)$ we construct has row weight $\mathsf{rajcode}(w)$ and column weight $\mathsf{rajcode}(w^{-1})$. 

        By Theorem~\ref{T: PSW} and the monomial-positive expansion of \(\fCM_w(x;y)\) in terms of marked bumpless pipedreams, there is at most one marked bumpless pipedream in \(\MBPD(w)\) with row weight \(\rajcode(w)\) and column weight \(\rajcode(w^{-1})\). Since $\widehat{\mathrm{D}}(w)$ constructed above has precisely these two weights, it is the unique maximal marked bumpless pipedream of \(w\).
    \end{proof}
	
	\bibliographystyle{alpha}
	\bibliography{citation}{}
\end{document}